\documentclass{amsart}

\usepackage{graphicx}

\usepackage[defaultlines=4,all]{nowidow}
    
\usepackage{xcolor}
\definecolor{dark-blue}{rgb}{0.15,0.15,0.4}

\usepackage{ pbox }

\usepackage{ amsmath }
\usepackage{ amsthm, thmtools }
\usepackage[colorlinks=true, linkcolor=black, citecolor=black, urlcolor=dark-blue, anchorcolor=dark-blue]{ hyperref }

\declaretheorem[style=plain,numberwithin=section]{theorem}
\declaretheorem[style=plain,numberlike=theorem]{lemma}
\declaretheorem[style=plain,numberlike=theorem]{proposition}
\declaretheorem[style=plain,numberlike=theorem]{corollary}
\declaretheorem[style=remark,numberlike=theorem]{remark}
\declaretheorem[style=definition,numberlike=theorem]{definition}
\declaretheorem[style=definition,numberlike=theorem]{example}

\makeatletter
\newcommand{\proofsubheading}[1]{%
  \par
  \addvspace{\medskipamount}%
  \noindent\emph{#1}\par\nobreak
  \addvspace{\smallskipamount}%
  \@afterheading
}
\makeatother

\usepackage{tikz}
\usepackage{tikz-cd}
\usetikzlibrary{shapes.misc} 

\usepackage{ mathtools }

\usepackage{ letltxmacro }

\usepackage[shortlabels]{ enumitem }

\usepackage[section]{ placeins } % ensures floats stay inside the same section
\usepackage[labelformat=simple]{ subcaption }

\usepackage{ amssymb }
\usepackage{ stmaryrd }
\usepackage{ blkarray }

\usepackage{ mathdots }

\usepackage{ bbm } 

\DeclareMathAlphabet{\mathcal}{OMS}{cmsy}{m}{n}
\newcommand{\G}{\mathcal{G}}
\newcommand{\Hscr}{\mathcal{H}}

\usepackage{ xparse }
\NewDocumentCommand\set{s m}{%
    \IfBooleanTF#1%
    {\left\{ #2 \right\}}%
    {\{#2\}}%
}
\NewDocumentCommand\setbuild{s m m}{%
    \IfBooleanTF#1%
    {\ensuremath{\left\{\, #2 \, \middle| \, #3 \,\right\}}}%
    {\ensuremath{\{\, #2 \, \mid \, #3 \,\}}}%
}

\DeclarePairedDelimiter{\floor}{\lfloor}{\rfloor}
\DeclarePairedDelimiter{\paren}{(}{)}
\DeclarePairedDelimiter{\abs}{\lvert}{\rvert}
\DeclarePairedDelimiter{\sqbra}{[}{]}
\makeatletter
\let\oldfloor\floor
\def\floor{\@ifstar{\oldfloor}{\oldfloor*}}
\let\oldparen\paren
\def\paren{\@ifstar{\oldparen}{\oldparen*}}
\let\oldabs\abs
\def\abs{\@ifstar{\oldabs}{\oldabs*}}
\let\oldsqbra\sqbra
\def\sqbra{\@ifstar{\oldsqbra}{\oldsqbra*}}
\makeatother

\newcommand{\F}{\mathbb{F}}

\newcommand{\iso}{\cong}

\usepackage{ faktor }
\LetLtxMacro{\origfaktor}{\faktor}
\DeclareDocumentCommand{\gapfaktor}{s m O{0.65} m O{-0.65}}{% \newfaktor[*]{#2}[#3]{#4}[#5] -> #2/#4
  \setbox0=\hbox{\ensuremath{#2}}% Store numerator
  \setbox2=\hbox{\ensuremath{#4}}% Store denominator
  \raisebox{#3\ht0}{\usebox0}% Numerator
  \mkern-5mu%
  {\rotatebox{-45}{\rule[#5\ht2]{0.4pt}{-#5\ht2+#3\ht0+\ht0-1pt}}} %tilted rule as a slash
  \mkern-4mu%
  \raisebox{#5\ht2}{\usebox2}% Denominator
}
\renewcommand{\faktor}{\gapfaktor}

\newcommand{\tensor}{\otimes}

\DeclareMathOperator{\Hom}{Hom}

\DeclareMathOperator{\Sym}{Sym}

\DeclareMathOperator{\soc}{soc}

\newcommand{\projfree}[1]{\Omega^0\left(#1\right)}

\DeclareMathOperator{\SL}{SL}

\DeclareMathOperator{\mean}{mean}

\makeatletter
\let\@@pmod\pmod
\DeclareRobustCommand{\pmod}{\@ifstar\@pmods\@@pmod}
\def\@pmods#1{\mkern4mu({\operator@font mod}\mkern 6mu#1)}
\makeatother

\newlength{\truelen}
\newcommand{\padbox}[3][c]{%
    \settowidth{\truelen}{\ensuremath{#2}}%
    \ifdim\truelen < #3%
        \makebox[#3][#1]{\ensuremath{#2}}%
    \else%
        \ensuremath{#2}%
    \fi%
}

\newlength{\padlen}
\newcommand{\pad}[1][]{\padbox[c]{#1}{\padlen}}
\newlength{\edgepadlen}
\newcommand{\edgepad}[1][]{\padbox[c]{#1}{\edgepadlen}}

\newlength{\minlen}
\newcommand{\flexbox}[2]{%
    \settowidth{\minlen}{\ensuremath{#2}}%
    \padbox[l]{#1}{\minlen}%
}

\newcommand{\longsubsum}[3]{%
    \flexbox{%
        \smashoperator[r]{#1_{#2}}\;#3%
    }{%
        {\scriptstyle#2}%
    }%
}

\DeclareMathOperator{\im}{im}

\newcommand{\blank}{{-}}

\newcommand{\one}[1]{\left[#1\right]}

\renewcommand{\epsilon}{\varepsilon}
\renewcommand{\phi}{\varphi}
\renewcommand{\leq}{\leqslant}
\renewcommand{\geq}{\geqslant}
\renewcommand{\emptyset}{\varnothing}

\newcommand{\p}{\mathfrak{p}}

\newcommand{\nmstring}[2]{\langle#1,#2\rangle}
\newcommand{\nmstr}{\nmstring{n}{m}}

\newcommand{\multiplicity}[2]{[#2 : #1]}

\newcommand{\simpledivide}[3]{\text{\(V_{#1}\) is a summand of \(V_{#2} \tensor V_{#3}\)}}

\newcommand{\algclosedfield}{k}
\newcommand{\primesubfield}{\F_p}

\title[A random walk on indecomposable representations of \(\SL_2(\primesubfield)\)]
{A random walk on the indecomposable \\ summands of tensor products of modular representations of \(\SL_2(\primesubfield)\)}

\usepackage{ amsaddr }
\author{Eoghan McDowell}
\address{Department of Mathematics, Royal Holloway, University of London \\ Egham, Surrey, TW20~0EX, United Kingdom}
\email{\textcolor{dark-blue}{\href{mailto:eoghan.mcdowell.2018@rhul.ac.uk}{eoghan.mcdowell.2018@rhul.ac.uk}}}


\makeatletter
\@namedef{subjclassname@2010}{\textup{2010} Mathematics Subject Classification}
\makeatother


\begin{document}

\begin{abstract}
In this paper we introduce a novel family of Markov chains on the simple representations of \(\SL_2(\F_p)\) in defining characteristic, defined by tensoring with a fixed simple module and choosing an indecomposable non-projective summand. %according to a specified weighting.
We show these chains are reversible and find their connected components and their stationary distributions.
We draw connections between the properties of the chain and the representation theory of \(\SL_2(\primesubfield)\), emphasising symmetries of the tensor product.
We also provide an elementary proof of the decomposition of tensor products of simple \(\SL_2(\primesubfield)\)-representations.

\end{abstract}

\maketitle

\section{Introduction}

This paper introduces a novel family of Markov chains on the simple representations of \(\SL_2(\F_p)\) in defining characteristic.
The chains are defined by, beginning at some simple module, tensoring by a fixed simple module and choosing a non-projective indecomposable summand with probability depending on a weighting given to each indecomposable module.

We analyse properties of this Markov chain in correspondence with the representation theory of \(\SL_2(\F_p)\), and in particular with the tables of multiplicities of simple modules as indecomposable summands of tensor products.
We show that our chain is reversible as a result of these tables being symmetric about the diagonal.
We show that, depending on the parity of the dimension of the fixed simple module, our chain is either irreducible with period \(2\) or consists of two irreducible aperiodic chains, as a result of these tables having rotational symmetry.
We also compute stationary distributions and mixing times.

If we were to replace \(\SL_2(\F_p)\) with an arbitrary group, it may not be clear how many states our chain has, or whether that number is finite.
For \(\SL_2(\F_p)\), however, we can deduce that there are only finitely many states even before calculating any decompositions: since \(\SL_2(\F_p)\) has a cyclic Sylow \(p\)-subgroup, all its simple representations are algebraic (in the sense of satisfying a polynomial over \(\mathbb{Z}\) in the Green ring), and hence their tensor powers collectively have only finitely many summands \cite[Lemma 1.1, Corollary 1.6]{craven2008thesis}.
We will see that furthermore the states are precisely the non-projective simple representations.

The study of this random walk is inspired by Benkart--Diaconis--Liebeck--Tiep \cite{benkart2020tensor}, who consider chains defined by choosing composition factors of tensor products.
Our walk differs by considering indecomposable summands rather than composition factors, and by permitting tensoring by any choice of fixed simple module (whereas Benkart--Diaconis--Liebeck--Tiep focus on tensoring by the natural module).

In our definition we ignore projective summands because tensoring by a projective module always yields a projective module, and so if they were included the projective modules would form an absorbing set.
Our definition allows us to consider a recurrent chain on the (non-projective) simple modules.
If we were to replace \(\SL_2(\primesubfield)\) with an algebraic group \(\SL_2(k)\), then this ``reduced'' tensor product would be the tensor product of a so-called fusion category of the kind studied in \cite[Section 2]{andersen1995}, and our random walk can be described in terms of fusion rules.
In this setting, the useful symmetries of the table of multiplicities are given by \cite[Axiom 3, p.~183]{mathieu2000} or by the modular Verlinde formula \cite[Theorem 9.5]{mathieu2000}.

In order to compute the transitions of our random walk, we require knowledge of the decomposition of tensor products of simple modules into indecomposable modules.
Such a decomposition is known as a Clebsch--Gordan rule;
the statement of the rule in our case is given in \autoref{thm:CG}.
This rule can be found by inductively tensoring by the natural representation of \(\SL_2(\F_p)\) (see \cite[(5.5) and (6.3)]{GLOVER1978425} or \cite[Corollary 1.2(a) and Proposition 1.3(c)]{Kouwenhoven1990}), or from tilting theory (see \cite[Lemma 4]{erdmann_henke_2002}).

We give our own proof of the Clebsch--Gordan rule in \autoref{section:sess}.
Our proof is novel and elementary, and is included for the purpose of making our analysis self-contained and demonstrating that the machinery of tilting theory is not required.
Our proof introduces a new family of maps (\autoref{def:lambda}) which exhibit explicit submodules and quotients of the tensor products and identifies cases in which they split; this results in a proof which finds the projective summands more efficiently than \cite{GLOVER1978425} and \cite{Kouwenhoven1990} without the theory required by \cite{erdmann_henke_2002}.
We give a convenient pictorial way to compute these decompositions in \autoref{eg:tensor_of_simples}.

The structure of this paper is as follows.
The simple modules are introduced in \autoref{section:background}.
Our proof of the Clebsch--Gordan rule is given in \autoref{section:sess}.
\autoref{section:matrix} studies the graph on which our Markov chain is a random walk, and
\autoref{section:random_walks} studies the Markov chain itself.

Throughout this paper, \(p\) denotes a prime, \(\algclosedfield\) denotes an algebraically closed field of characteristic \(p\),
and \(\primesubfield\) denotes its prime subfield.
All modules are assumed to be finite-dimensional over \(k\).
We write \(G = \SL_2(\F_p)\), the group of \(2 \times 2\) matrices with entries in \(\primesubfield\) with determinant \(1\).
We make use of Iverson bracket notation, and the notation \([r] = \set{1, 2, \ldots, r}\).
We also introduce the following notation for a family of sets that will index the summands of tensor products.
\begin{definition}
For \(n \geq m \geq 1\), let the \emph{\((n,m)\)-string} be the set
\[
\nmstring{n}{m} = \set{n+m-1, n+m-3, \ldots, n-m+3, n-m+1},
\]
and let \(\nmstring{n}{0} = \emptyset\).
\end{definition}

\section{Simple representations}
\label{section:background}

In this section we give an explicit model for the simple representations of \(G\), and state some basic facts about them and their projective covers.

\begin{definition}
For \(n \geq 1\),
let \(V_n\) be the \(n\)-dimensional \(\algclosedfield G\)-module consisting of homogeneous polynomials over \(\algclosedfield\) of degree \(n-1\) in two variables \(X\) and \(Y\), with \(G\)-action given by
\[
    \begin{psmallmatrix}
            a & b \\ c & d \\
    \end{psmallmatrix}
    f(X, Y) = f(aX+cY, bX+dY).
\]
Let \(P_n\) be the projective cover of \(V_n\).
\end{definition}

Note that \(V_1\) is the trivial module, \(V_2\) is isomorphic to the natural \(\algclosedfield G\)-module, and that \(V_n \iso \Sym^{n-1} V_2\).

The representations \(V_1, \ldots, V_p\) are simple \cite[pp.~14--16]{alperin_1986}.
Furthermore, the set \(\setbuild{V_n}{1 \leq n \leq p}\) is a complete set of simple \(\algclosedfield G\)-modules up to isomorphism, since the number of \(p\)-regular conjugacy classes in \(G\) is \(p\).
In particular, there is a unique simple \(\algclosedfield G\)-module of each dimension less than or equal to \(p\), and so the simple modules are self-dual.
Also, it follows that the set \(\setbuild{P_n}{1 \leq n \leq p}\) is a complete set of projective indecomposable \(\algclosedfield G\)-modules.

The projective indecomposable \(\algclosedfield G\)-modules are constructed in \cite[pp.~48--52]{alperin_1986} (using the special case \(m=2\) of our \autoref{prop:mu_ses}).
We describe them here.
Firstly, \(P_p \iso V_p\) is projective and simple.
When \(p=2\), there is only one other projective indecomposable module: \(P_1\), which is of Loewy length \(2\) (and hence has composition factors only \(V_1\)).
For \(p > 2\), all other projective indecomposable modules have Loewy length \(3\), and so the only structural information which is missing is their hearts.
The heart of \(P_1\) is \(V_{p-2}\), the heart of \(P_{p-1}\) is \(V_{2}\), and for \(2 \leq n \leq p-2\) the heart of \(P_n\) is \(V_{p-n-1} \oplus V_{p-n+1}\); these structures are illustrated below.
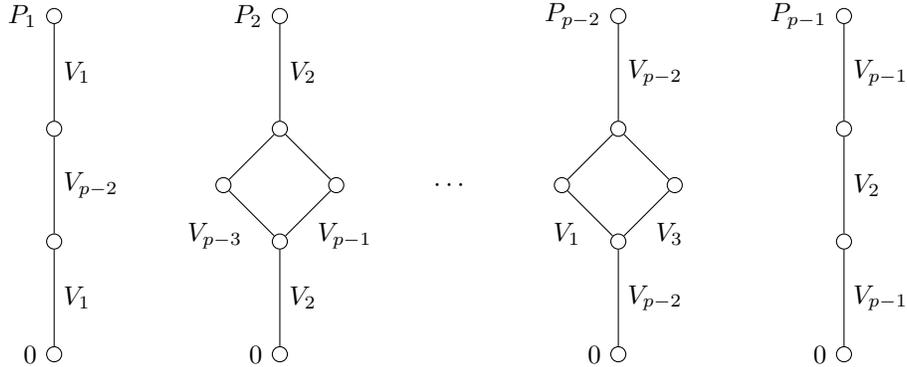
\tikzset{
    submodule/.style={circle,draw, minimum size=2mm, inner sep=0}
}
\begin{center}
\begin{tikzpicture}[
    auto,
    node distance=1.5cm,
]
    \node (top1) at (-5.25, 0) [submodule, label=left:\(P_1\)] {};
    \node (rad1) [below of=top1, submodule] {};
    \node (soc1) [below of=rad1, submodule] {};
    \node (zero1) [below of=soc1, submodule, label=left:\(0\)] {};
    
    \draw (top1) -- node {\(V_1\)} (rad1);
    \draw (rad1) -- node {\(V_{p-2}\)}(soc1);
    \draw (soc1) -- node {\(V_1\)} (zero1);

    \node (top2) at (-2.25,0) [submodule, label=left:\(P_2\)] {};
    \node (rad2) [below of=top2, submodule] {};
    \node (left2) at (-3, -2.25) [submodule] {};
    \node (right2) at (-1.5, -2.25) [submodule] {};
    \node (soc2) [below of=rad2, submodule] {};
    \node (zero2) [below of=soc2, submodule, label=left:\(0\)] {};
    
    \draw (top2) -- node {\(V_2\)} (rad2);
    \draw (rad2) -- (left2);
    \draw (rad2) -- (right2);
    \draw (left2) -- node [pos=0.5, swap] {\(V_{p-3}\)} (soc2);
    \draw (right2) -- node [pos=0.5] {\(V_{p-1}\)} (soc2);
    \draw (soc2) -- node {\(V_2\)} (zero2);
    
    \path (-0.5,-2.250) -- node[auto=false]{\(\cdots\)} (0.5,-2.25);
    
    \node (top3) at (2.25,0) [submodule, label=left:\(P_{p-2}\)] {};
    \node (rad3) [below of=top3, submodule] {};
    \node (left3) at (1.5, -2.25) [submodule] {};
    \node (right3) at (3, -2.25) [submodule] {};
    \node (soc3) [below of=rad3, submodule] {};
    \node (zero3) [below of=soc3, submodule, label=left:\(0\)] {};
    
    \draw (top3) -- node {\(V_{p-2}\)} (rad3);
    \draw (rad3) -- (left3);
    \draw (rad3) -- (right3);
    \draw (left3) -- node [pos=0.5, swap] {\(V_{1}\)} (soc3);
    \draw (right3) -- node [pos=0.5] {\(V_{3}\)} (soc3);
    \draw (soc3) -- node {\(V_{p-2}\)} (zero3);
    
    \node (top4) at (5.25,0) [submodule, label=left:\(P_{p-1}\)] {};
    \node (rad4) [below of=top4, submodule] {};
    \node (soc4) [below of=rad4, submodule] {};
    \node (zero4) [below of=soc4, submodule, label=left:\(0\)] {};
    
    \draw (top4) -- node {\(V_{p-1}\)} (rad4);
    \draw (rad4) -- node {\(V_{2}\)}(soc4);
    \draw (soc4) -- node {\(V_{p-1}\)} (zero4);
\end{tikzpicture}
\end{center}
Note that \(P_1\) and \(P_p\) are both \(p\)-dimensional, while all other projective indecomposable \(\algclosedfield G\)-modules are \(2p\)-dimensional.

% We can now write down the Cartan matrix.
% It is most convenient to give the simple modules and their covers the ordering
% \[
%     V_1, V_{p-2}, V_3, \ldots, V_{\frac{p+\epsilon}{2}}, V_{p-1}, V_2, V_{p-3}, \ldots, V_{\frac{p-\epsilon}{2}}, V_p
% \]
% where \(\epsilon \in \set{\pm 1}\) and \(\epsilon \equiv p \pmod{4}\).
% For \(p=2\), the Cartan matrix is simply \(\begin{psmallmatrix}
% 2 & 0 \\ 0 & 1
% \end{psmallmatrix}\).
% For \(p > 2\), let \(C\) be the \(\frac{p-1}{2} \times \frac{p-1}{2}\) matrix
% \[
% \renewcommand{\arraystretch}{1.1}
% \begin{pmatrix}
%   2  &  1       &        &        &      \\
%   1  &  2       &   1    &        &      \\[-3pt]
%      &  \ddots  & \ddots & \ddots &      \\
%      &          &    1   &    2   &   1  \\
%      &          &        &    1   &   3  \\
% \end{pmatrix}
% \]
% where \(C = \begin{pmatrix}
% 3
% \end{pmatrix}\) when \(p=3\). Then the Cartan matrix, in block diagonal form, is
% \[
% \begin{pmatrix}
% C &   &    \\
%   & C &    \\
%   &   &  1 \\
% \end{pmatrix}.
% \]

\section{Clebsch--Gordan Rule}
\label{section:sess}

In this section we provide an elementary proof of the decomposition of tensor products of simple \(kG\)-modules into indecomposable summands.
Our approach is to consider two families of maps: the multiplication maps \(\mu\), and a novel family of maps \(\lambda\).

We begin by considering the multiplication maps.

\begin{definition}
Let \(\mu \colon V_n \tensor V_m \to V_{n+m-1}\) be the multiplication map, defined by \(\algclosedfield\)-linear extension of \(\mu(f \tensor g) = fg\).
The dependence of \(\mu\) on \(n\) and \(m\) is suppressed, since it is always clear from context.
\end{definition}

It is easily seen that \(\mu\) is surjective and \(G\)-equivariant.
The following result identifying the kernel of \(\mu\) is well-known (see \cite[(5.1)]{GLOVER1978425}, or for the case \(m=2\) \cite[Lemma 5, p.~50--51]{alperin_1986} or \cite[Proposition 1.2(a)]{Kouwenhoven1990}).

\begin{proposition}\label{prop:mu_ses}
%Suppose \(G \leq \SL_2(\algclosedfield)\) and 
Suppose \(n,m \geq 2\).
Then the kernel of \(\mu\) is isomorphic to \(V_{n-1} \tensor V_{m-1}\), and hence there is a short exact sequence
\[
\begin{tikzcd}
    0 \arrow[r] 
        & V_{n-1} \tensor V_{m-1} \arrow[r]
        & V_n \tensor V_m \arrow[r, "\mu"]
        & V_{n+m-1} \arrow[r] 
        & 0.
\end{tikzcd}
\]
\end{proposition}

\begin{proof}[Proof\nopunct]
Consider the map \(\theta \colon V_{n-1} \tensor V_{m-1} \to V_n \tensor V_m\) defined by \(\algclosedfield\)-linear extension of \(\theta(f \tensor g) = Xf \tensor Yg - Yf \tensor Xg\).
It is easily verified directly that \(\theta\) is \(\SL_2(\algclosedfield)\)-equivariant, and clearly \(\im \theta \leq \ker \mu\).
% for \(t = \begin{psmallmatrix}
% a & b \\ c & d
% \end{psmallmatrix}
%  \in \SL_2(\algclosedfield)\), we have
% \begin{align*}
%     t\theta(f \tensor g)
%         &= t(Xf \tensor Yg - Yf \tensor Xg) \\
%         &= (a X+ c Y)tf \tensor (b X+ d Y)tg - (b X+ d Y)tf \tensor (a X+c Y)tg \\
%         &= (a d - b c)Xtf \tensor Ytg - (ad - bc)Ytf \tensor Xtg \\
%         &= \det(t)(Xtf \tensor Ytg - Ytf \tensor Xtg) \\
%         &= Xtf \tensor Ytg - Ytf \tensor Xtg \\
%         &= \theta( t( f \tensor g) ).
% \end{align*}
%It is easy to see that \(\im \theta \leq \ker \mu\).
Because \(\mu\) is surjective, we have that \(\dim(\ker \mu) = \dim(V_n \tensor V_m) - \dim(V_{n+m-1}) = \dim(V_{n-1} \tensor V_{m-1})\), and so it remains only to show that \(\theta\) is injective.

Let \(e_{i,j} = X^i Y^{n-2-i} \tensor X^j Y^{m-2-j} \in V_{n-1} \tensor V_{m-1}\),
so that \(\setbuild{e_{i,j}}{0 \leq i \leq n-2,\, 0 \leq j \leq m-2}\) is a linear basis for \(V_{n-1} \tensor V_{m-1}\).
For \(0 \leq r \leq n+m-4\), let \(U_r = \langle\, e_{i,j} \mid i+j=r \,\rangle_\algclosedfield \subseteq_\algclosedfield V_{n-1} \tensor V_{m-1}\), and note that as vector spaces \(V_{n-1} \tensor V_{m-1} = \bigoplus_{r=0}^{n+m-4} U_r\).
Similarly,
let \(f_{i,j} = X^i Y^{n-1-i} \tensor X^j Y^{m-1-j} \in V_{n} \tensor V_m\),
so that \(\setbuild{f_{i,j}}{0 \leq i \leq n-1,\, 0 \leq j \leq m-1}\) is a linear basis for \(V_{n} \tensor V_{m}\).
For \(0 \leq r \leq n+m-2\), let \(W_r = \langle\, f_{i,j} \mid i+j=r \,\rangle_\algclosedfield \subseteq_\algclosedfield V_{n} \tensor V_{m}\), and note that as vector spaces \(V_{n} \tensor V_{m} = \bigoplus_{r=0}^{n+m-2} W_r\).

Observe that \(\theta(e_{i,j}) = f_{i+1,j} - f_{i,j+1}\).
Then \(\theta(U_r) \subseteq_\algclosedfield W_{r+1}\), 
and thus it suffices to show that \(\theta|_{U_r}\) is injective for each \(0 \leq r \leq n+m-4\).

Fix \(r\) in this range, and let \(i_0 = \max\set{0,r-(m-2)}\) and \(j_0 = \max\set{0,r-(n-2)}\)
so that \(U_r = \langle\, e_{i,r-i} \mid i_0 \leq i \leq r-j_0 \,\rangle_\algclosedfield\).
% Then the images under \(\theta\) of these basis vectors for \(U_r\) are as follows.
% \begin{align*}
%     \theta(e_{i_0, r-i_0})           &= f_{i_0 + 1, r+1 - (i_0 + 1)} - f_{i_0, r+1 - i_0} \\
%     \theta(e_{i_0 + 1, r-(i_0 + 1)}) &= f_{i_0 + 2, r+1 - (i_0 + 2)} - f_{i_0 + 1, r+1 - (i_0+1)} \\
%     \theta(e_{i_0 + 2, r-(i_0 + 2)}) &= f_{i_0 + 3, r+1 - (i_0 + 3)} - f_{i_0 + 2, r+1 - (i_0+2)} \\
%         &\;\vdots \\
%     \theta(e_{r-(j_0 + 1), j_0 + 1}) &= f_{r+1 - (j_0+1), j_0+1} - f_{r+1 - (j_0+2), j_0+2} \\
%     \theta(e_{r-j_0, j_0})           &= f_{r+1 - j_0, j_0} - f_{r+1 - (j_0+1), j_0 + 1}
% \end{align*}
Then the \((r - i_0 - j_0 + 1) \times (r - i_0 - j_0)\) matrix representing \(\theta\) with respect to the given bases is
\[
\renewcommand{\arraystretch}{1.2}
\begin{pmatrix}
\phantom{-}1    &               &       &           \\
        -1      &\phantom{-}1   &       &           \\
                & -1            &\raisebox{-2pt}{\ensuremath{\smash{\ddots}}} &           \\
                &               &\raisebox{-2pt}{\ensuremath{\smash{\ddots}}} &\phantom{-}1 \\
                &               &       & -1       \\
\end{pmatrix},
\]
which is of full (column) rank.
Thus \(\theta|_{U_r}\) is injective as required.
\end{proof}

% \begin{remark}
% Unlike \(\mu\), the map \(\theta\) is \emph{not} \(\GL_2(\algclosedfield)\)-equivariant.
% Since \(t \theta(f \tensor g) = \det(t) \theta(t( f \tensor g))\) for \(t \in \GL_2(\algclosedfield)\), we see that \(\theta\) is not \(G\)-equivariant for any subgroup \(G\) which contains a matrix with determinant not equal to \(1\).
% For an extension of this proposition to such subgroups, see \cite[(5.1)]{GLOVER1978425}.
% \end{remark}

Applying \autoref{prop:mu_ses} inductively, we obtain the following.

\begin{corollary}\label{cor:nm_comp_factors}
Suppose \(1 \leq m \leq n\).
Then \(V_n \tensor V_m\) has a filtration
\[
    0 = M_m \subseteq M_{m-1}
    \subseteq \cdots \subseteq M_{1} \subseteq M_0 = V_n \tensor V_m
\]
where \(M_i \iso V_{n-i} \tensor V_{m-i}\) and \(\faktor{M_i}{M_{i+1}} \iso V_{n+m-1-2i}\). 
\end{corollary}

% \begin{proof}[Proof\nopunct]
% By induction on \(m\). The case \(m=1\) is immediate.
% For \(m \geq 2\), the short exact sequence involving \(\mu\) gives that 
% there is \(M_1 \subseteq V_n \tensor V_m\) such that \begin{align*}
%     M_1 &\iso V_{n-1} \tensor V_{m-1}
% \intertext{and}
%     \faktor{V_n \tensor V_m}{M_1} &\iso V_{n+m-1}.
% \end{align*}
% Applying the inductive hypothesis to \(M_1\) gives the rest of the filtration.
% \end{proof}

% \begin{remark}
% The proof of \autoref{prop:mu_ses} holds equally well if \(\algclosedfield\) is of characteristic \(0\).
% In this case the simple modules are also projective and so the short exact sequences split, and we obtain \(V_n \tensor V_m \iso \bigoplus_{i \in \nmstr} V_i\) (recovering the well-known Clebsch--Gordan rule for \(\SU_2(\C)\)).
% The same decomposition is obtained when \(G \leq \SL_2(\algclosedfield)\) is finite with \(p \notdivides \abs{G}\).
% \end{remark}

We now introduce a novel family of maps, which generalise the map \(\delta\) defined in \cite[(5.2)]{GLOVER1978425} (corresponding to \(n=1\) below).
These maps allow us to see the inclusion of the bottom layer of the above filtration into \(V_n \tensor V_m\), and they split in more cases than \(\mu\) does.

\begin{definition}\label{def:lambda}
For \(n \geq 1\) and \(m \geq 2\),
let \(\lambda \colon V_n \tensor V_m \to V_{n+1} \tensor V_{m-1}\) be the map defined by \(\algclosedfield\)-linear extension of
\[
    \lambda( f \tensor g ) =
        Xf \tensor \frac{\partial g}{\partial X} + Yf \tensor \frac{\partial g}{\partial Y}.
\]
The dependence of \(\lambda\) on \(n\) and \(m\) is suppressed, since it is always clear from context.
\end{definition}

It is routine to verify directly that \(\lambda\) is \(G\)-equivariant.

\begin{lemma}\label{lemma:lambda_surjective}
Suppose \(n \geq m\) and \(2 \leq m \leq p\). Then the map \(\lambda\) is surjective.
\end{lemma}

\begin{proof}[Proof\nopunct]
Let \(f \in V_{n+1}\), \(g \in V_{m-1}\) be monomials, and we aim to show that \(f \tensor g \in \im \lambda\).
Without loss of generality, suppose \(X\) occurs with a higher power than \(Y\) in the product \(fg\).
Then \(X\) must occur with a positive power in \(f\), and so \(\frac{1}{X}f\) is a monomial in \(V_{n}\).
Let \(j\) be the power of \(X\) in \(g\).
Observe
\[
    \lambda\left( \tfrac{1}{X}f \tensor Xg\right)
        = (j+1) f \tensor g + \frac{Y}{X}f \tensor X \frac{\partial g}{\partial Y}.
\]
By inducting downward on \(j\), we may assume that \(\frac{Y}{X}f \tensor X \frac{\partial g}{\partial Y} \in \im \lambda\) (with base case having \(g = X^{m-2}\) and hence \(\frac{\partial g}{\partial Y} = 0\)).
Since also \(j+1\) is invertible, the result follows.
%
% Let \(f = X^i Y^{i'} \in V_{n+1}\), \(g = X^j Y^{j'} \in V_{m-1}\) be monomials.
% We have \(i+ i'+ j + j' = n+m-2 \geq 2(m-1)\), and hence either \(i+j \geq m-1\) or \(i' + j' \geq m-1\).
% We show that \(f \tensor g \in \im \lambda\) by downward induction on \(j\) whenever \(i + j \geq m-1\); then by analogy the same holds whenever \(i' + j' \geq m-1\).

% Note first that if \(i + j \geq m-1\), then \(i \geq 1\) (since \(0 \leq j \leq m-2\)) and so \(\frac{1}{X}f\) is a monomial (in \(V_n\)).

% If \(j = m-2\),
% then \(g = X^{m-2}\) so \(\frac{\partial (Xg)}{\partial X} = (m-1)X^{m-2}\) and \(\frac{\partial g}{\partial Y} = 0\). Then
% \[
%     \lambda\left(\tfrac{1}{X} f \tensor Xg\right)
%         = (m-1) f \tensor g
% \]
% and \(m-1\) is invertible (since \(2 \leq m \leq p\)), so \(f \tensor g \in \im \lambda\).

% Now suppose \(0 \leq j < m-2\). Then
% \[
%     \lambda\left( \tfrac{1}{X}f \tensor Xg\right)
%         = (j+1) f \tensor g + \frac{Y}{X}f \tensor X \frac{\partial g}{\partial Y}.
% \]
% But by the inductive hypothesis \(\frac{Y}{X}f \tensor X \frac{\partial g}{\partial Y} \in \im \lambda\) (since \(X \frac{\partial g}{\partial Y}\) has a higher power of \(X\) than \(g\), and the sum of the powers of \(X\) in \(\frac{Y}{X}f\) and \(X \frac{\partial g}{\partial Y}\) is \(i+j \geq m-1\)).
% Then since \(j+1\) is invertible, we have \(f \tensor g \in \im \lambda\).
\end{proof}

\begin{proposition}\label{prop:lambda_ses}
%Suppose \(G \leq \SL_2(k)\) and 
Suppose \(n \geq m\) and \(2 \leq m \leq p\).
Then the kernel of \(\lambda\) is isomorphic to \(V_{n-m+1}\), and hence there is a short exact sequence
\[
\begin{tikzcd}
    0 \arrow[r] 
        & V_{n-m+1} \arrow[r]
        & V_n \tensor V_m \arrow[r, "\lambda"]
        & V_{n+1} \tensor V_{m-1} \arrow[r] 
        & 0.
\end{tikzcd}
\]
\end{proposition}

\begin{proof}[Proof\nopunct]
Define \(G\)-equivariant variations on the multiplication map by
\begin{alignat*}{3}
\mu^{(r)} \colon
    V_{n_1} \tensor V_{m_1} \tensor &\cdots \tensor V_{n_r} \tensor V_{m_r} \ &\to&&\ V_{N-(r-1)} &\tensor V_{M-(r-1)} \\
    f_1 \tensor g_1 \tensor &\cdots \tensor f_r \tensor g_r &\mapsto&& f_1 \cdots f_r &\tensor g_1 \cdots g_r
\end{alignat*}
extended \(k\)-linearly, where \(N = \sum_{i=1}^r n_i\) and \(M = \sum_{i=1}^r m_i\).
Let \(g_m \in V_m \tensor V_m\) be the element
\begin{align*}
g_m &= \mu^{(m-1)} \left( (X \tensor Y - Y \tensor X) \tensor \cdots \tensor (X \tensor Y - Y \tensor X) \right) \\
    &= \sum_{i=0}^{m-1} (-1)^{m-1-i} \binom{m-1}{i} X^i Y^{m-1-i} \tensor X^{m-1-i} Y^i.
\end{align*}
By the first expression, it is clear that \(t g_m = (\det t)^{m-1} g_m = g_m\) for any \(t \in G\).

Now define a \(k\)-linear map \(\eta \colon V_{n-m+1} \to V_n \tensor V_m\) by
\begin{align*}
    \eta(f)
        &= \mu^{(2)}(f \tensor 1 \tensor g_m) \\
        &= \sum_{i=0}^{m-1} (-1)^{m-1-i} \binom{m-1}{i} fX^i Y^{m-1-i} \tensor X^{m-1-i} Y^i.
\end{align*}
We see from the first expression that for any \(t \in G\), we have \(t\eta(f) = (\det t)^{m-1} \eta(tf) = \eta(tf)\). That is, \(\eta\) is \(G\)-equivariant.
Clearly the second expression above is zero if and only if \(f = 0\), so \(\eta\) is injective and \(V_{n-m+1} \iso \im \eta\).

It now suffices to show \(\im \eta \iso \ker \lambda\).
Using the second expression for \(\eta(f)\) above, it is a routine manipulation of binomial coefficients to verify that \(\im \eta \leq \ker \lambda\).
Counting dimensions then yields the result.
%
%
% Furthermore,
% \begin{align*}
% \lambda\eta(f)
%     &= \sum_{i=0}^{m-2} (-1)^{m-1-i} \binom{m-1}{i}(m-1-i) fX^{i+1} Y^{m-1-i} \tensor X^{m-2-i} Y^i \\&\qquad\quad + \sum_{i=1}^{m-1} (-1)^{m-1-i} \binom{m-1}{i}i fX^{i}Y^{m-i} \tensor X^{m-1-i} Y^{i-1} \\
%     &= 0,
% \end{align*}
% where the final equality can be seen by replacing \(i\) with \(i-1\) in the first sum, and noting that \(\binom{m-1}{i}i = (m-1)\binom{m-2}{i-1} = \binom{m-1}{i-1}(m-i)\). Thus \(V_{n-m+1} \iso \im \eta \leq \ker \lambda\).
%
% Since \(n \geq m\) and \(2 \leq m \leq p\), by \autoref{lemma:lambda_surjective} we have that \(\lambda\) is surjective, and then by counting dimensions we have \(V_{n-m+1} \iso \ker \lambda\).
\end{proof}

\begin{theorem}[Clebsch--Gordan rule for \(\SL_2(\primesubfield)\) in characteristic \(p\)]\label{thm:CG}
Suppose \(1 \leq m \leq n \leq p\).
Recall we write \(\nmstring{n}{m} = \set{n+m-1, n+m-3, \ldots, n-m+3, n-m+1}\).
Then
\[
V_n \tensor V_m 
\iso
    \longsubsum{\bigoplus}{\substack{i \in \nmstr \cap [p]\\ 2p-i \notin \nmstr }}{V_i}
        \oplus
    \longsubsum{\bigoplus}{\substack{i \in \nmstr \cap [p]\\ 2p-i \in \nmstr}}{P_i}
        \oplus
    \one{n = m = p}V_p.
\]
\end{theorem}

% \begin{remark}
% We make several immediate observations about the tensor product of simple modules \(V_n\) and \(V_m\) (where \(1 \leq m \leq n \leq p\)):
% \begin{enumerate}[(i)]
%     \item
% all non-projective summands of \(V_n \tensor V_m\) are simple;
%     \item
% \(V_n \tensor V_m\) is semisimple if and only if \(n+m \leq p+1\), in which case \(V_n \tensor V_m \iso \bigoplus_{i \in \nmstr} V_i\), which is exactly the decomposition for analogously defined representations of \(\SU_2(\C)\) over \(\C\);
%     \item
% \(V_n \tensor V_m\) is projective if and only if \(n = p\), in which case \(V_p \tensor V_m \iso \bigoplus_{i \in \nmstring{p}{m}\cap[p]} P_i \oplus \one{m=p}V_p\);
%     \item
% in the sense of indecomposable summands, \(V_n \tensor V_m\) is multiplicity-free unless \(n=m=p\) (when \(V_p\) occurs with multiplicity \(2\), and all other indecomposable summands occur only once).
% \end{enumerate}
% \end{remark}

\begin{proof}[Proof\nopunct]
Our strategy is to establish two implications.
Implication (i) is that if the theorem holds for \((n+1,m-1)\), then it holds for \((n,m)\) (where \(2 \leq m \leq n \leq p-1\)), which we prove by showing that the short exact sequence involving \(\lambda\) splits in this case.
Implication (ii) is that if the theorem holds for \((p-1,m-1)\), then it holds for \((p,m)\) (where \(2 \leq m \leq p\)), which we prove using the short exact sequence involving \(\mu\).
With these implications, it suffices to show the theorem holds for \((n,1)\) for \(1 \leq n \leq p\) (as illustrated in the case \(p=7\) in \autoref{fig:logic_diagram}).
But these cases are trivial, since \(V_n \tensor V_1 \iso V_n\) (and \(P_p \iso V_p\)), so the theorem follows.

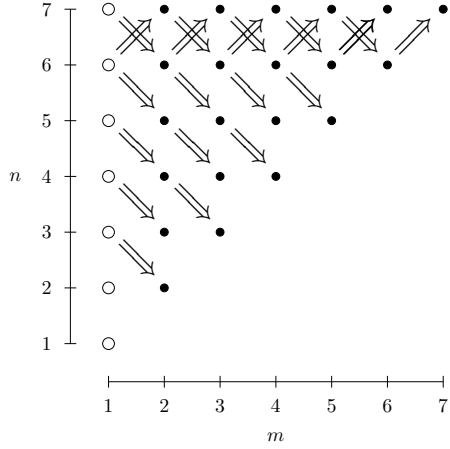
\begin{figure}[ht]\centering
\newcommand{\dotat}[2]{\draw[fill=black] (#1,#2) circle (2pt);}
\newcommand{\emptydotat}[2]{\draw[fill=white] (#1,#2) circle (3pt);}
\newcommand{\dr}[2]{\node at (#1+0.5,#2+0.5){\rotatebox{-45}{\scalebox{1.6}{\(\implies\)}}};}
\newcommand{\ur}[2]{\node at (#1+0.5,#2+0.5) {\rotatebox{45}{\scalebox{1.6}{\(\implies\)}}};}
\renewcommand{\t}{3pt}  %tick marks
\renewcommand{\l}{12pt} %labels
\renewcommand{\o}{9pt}  %offset axes
\resizebox{!}{60mm}{
\begin{tikzpicture}[x=1cm, y=1cm]
\draw (1,\o)--(7,\o);
\draw (\o,1)--(\o,7);

\foreach \i in {1,2,3,4,5,6,7} {\draw(\i,-\t+\o)--(\i,\t+\o); \node at (\i, -\l+\o) {\i};}
\foreach \j in {1,2,3,4,5,6,7} {\draw(-\t+\o,\j)--(\t+\o,\j); \node at (-\l+\o, \j) {\j};}

\node at (4,\o-28pt) {\(m\)};
\node at (\o-28pt,4) {\(n\)};

\foreach \j in {1,2,3,4,5,6,7} {\emptydotat{1}{\j}}
\foreach \j in {2,3,4,5} {\dr{1}{\j}}

\foreach \j in {2,3,4,5,6,7} {\dotat{2}{\j}};
\foreach \j in {3,4,5} {\dr{2}{\j}};

\foreach \j in {3,4,5,6,7} {\dotat{3}{\j}};
\foreach \j in {4,5} {\dr{3}{\j}};

\foreach \j in {4,5,6,7} {\dotat{4}{\j}};
\foreach \j in {5} {\dr{4}{\j}};

\foreach \j in {5,6,7} {\dotat{5}{\j}};
\foreach \j in {6,7} {\dotat{6}{\j}};
\foreach \j in {7} {\dotat{7}{\j}};

\foreach \i in {1,2,3,4,5} {\ur{\i}{6}; \dr{\i}{6}} \ur{5}{6} \ur{6}{6}

\end{tikzpicture}
}
\caption{An illustration, in the case \(p=7\), of how the implications we prove suffice to prove the entire theorem.
The dot in position \((n,m)\) represents the theorem holding for that pair of values, the hollow dots being the trivial cases with \(m=1\).
The arrows represent the implications we prove:
the downward-and-right arrows correspond to implication (i); the upward-and-right arrows to implication (ii).}
\label{fig:logic_diagram}
\end{figure}

\proofsubheading{Implication (i)}

Suppose the theorem holds for \((n+1,m-1)\) (where \(2 \leq m \leq n \leq p-1\)); that is,
\[
    V_{n+1} \tensor V_{m-1}
    \iso
    \longsubsum{\bigoplus}{\substack{i \in \nmstring{n+1}{m-1} \cap [p]\\ 2p-i \notin \nmstring{n+1}{m-1} }}{V_i}
        \oplus
    \longsubsum{\bigoplus}{\substack{i \in \nmstring{n+1}{m-1} \cap [p]\\ 2p-i \in \nmstring{n+1}{m-1}}}{P_i.}
\]
Observe that the proposed decomposition of \(V_{n} \tensor V_{m}\) differs from that of \(V_{n+1} \tensor V_{m-1}\) only by an additional summand of \(V_{n-m+1}\).
Thus to show the theorem holds for \((n,m)\), it suffices to show that the short exact sequence
\[
\begin{tikzcd}
    0 \arrow[r] 
        & V_{n-m+1} \arrow[r]
        & V_{n} \tensor V_{m} \arrow[r, "\lambda"]
        &
    V_{n+1} \tensor V_{m-1} \arrow[r]
        & 0
\end{tikzcd}
\]
from
\autoref{prop:lambda_ses} splits.

Let \(Q \iso \bigoplus_{{\substack{i \in \nmstring{n+1}{m-1} \cap [p]\\ 2p-i \in \nmstring{n+1}{m-1} }}} P_i\) be the projective part of \(V_{n+1} \tensor V_{m-1}\).
Then the projection of \(\lambda\) onto \(Q\) splits, and so there is a module \(W\) such that
\[
    V_{n} \tensor V_{m}
    \iso
    W
        \oplus
    Q
\]
and such that there is a short exact sequence
\[
\begin{tikzcd}
    0 \arrow[r] 
        & V_{n-m+1} \arrow[r]
        & W \arrow[r]
        &
    \bigoplus_{{\substack{i \in \nmstring{n+1}{m-1} \cap [p]\\ 2p-i \notin \nmstring{n+1}{m-1} }}} V_i \arrow[r] 
        & 0.
\end{tikzcd}
\]

It now suffices to show that this sequence splits.
Suppose, towards a contradiction, the sequence does not split.
Then \(W\), and hence \(V_{n} \tensor V_{m}\), has as an indecomposable summand some non-split extension \(T\) of \(V_{n-m+1}\) by a module with composition factors a nonempty subset of \(\setbuild{V_i}{i \in \nmstring{n+1}{m-1}\cap[p]}\).
This set of composition factors does not contain \(V_{n-m+1}\) itself, so \(T\) is not self-dual.
Furthermore, the dual of \(T\) is not a summand of \(W\), since \(V_{n-m+1}\) occurs only once as a composition factor of \(W\), and nor is it a summand of \(Q\), since \(V_{n-m+1}\) does not occur as the head of any of the projective summands of \(Q\).
Thus the dual of \(T\) is not a summand of \(V_{n} \tensor V_{m}\),
contradicting the self-duality of \(V_{n} \tensor V_{m}\).
So the sequence splits as required.

\proofsubheading{Implication (ii)}

Suppose the theorem holds for \((p-1,m-1)\) (where \(2 \leq m \leq p\)).
Then, using that
\(\nmstring{p-1}{m-1} \cap [p] = \paren{\nmstring{p}{m} \setminus \set{p+m-1}} \cap [p] = \nmstring{p}{m}\cap [p]\), we have
\[
    V_{p-1} \tensor V_{m-1}
        \iso V_{p-m+1} \oplus \longsubsum{\bigoplus}{\substack{i \in \nmstring{p}{m}\cap [p] \\ i \neq p-m+1}}{P_i.}
\]
Then by \autoref{prop:mu_ses} we have a short exact sequence
\[
\begin{tikzcd}
    0 \arrow[r] 
        & V_{p-m+1} \oplus\; \bigoplus_{\substack{i \in \nmstring{p}{m}\cap [p] \\ i \neq p-m+1}}\; P_i \arrow[r]
        & V_p \tensor V_m \arrow[r, "\mu"]
        & V_{p+m-1} \arrow[r] 
        & 0.
\end{tikzcd}
\]
Thus \(\bigoplus_{i \in \nmstring{p}{m}\cap[p]} V_i\) is isomorphic to a submodule of \(\soc(V_p \tensor V_m)\).
But since \(V_p\) is projective, so is \(V_p \tensor V_m\) (because the tensor product of a projective module with any other module is projective, as shown in \cite[Lemma 4, p.~47]{alperin_1986}).
Thus \(\bigoplus_{i \in \nmstring{p}{m}\cap[p]} P_i\) is isomorphic to a submodule of \(V_p \tensor V_m\).

The proof is completed by counting dimensions to show that this submodule is the entire tensor product, unless \(m = p\) when we must identify one additional summand.
The dimension counting is straightforward, recalling from \autoref{section:background} that the projective indecomposable \(\algclosedfield\SL_2(\primesubfield)\)-modules are \(2p\)-dimensional, except for \(P_1\) and \(P_p \iso V_p\) which are \(p\)-dimensional.
When \(m = p\), we find that exacly \(p\) dimensions are missing; since \(V_p \tensor V_p\) is projective, these \(p\) dimensions must be accounted for by an additional copy of either \(P_1\) or \(V_p\).
Since \(V_p\) is self-dual, we have \(V_p \tensor V_p \iso \Hom_\algclosedfield(V_p, V_p)\).
Noting that the direct sum of all trivial submodules of \(\Hom_\algclosedfield(V_p, V_p)\) is \(\Hom_{\algclosedfield G}(V_p, V_p)\), which is isomorphic to \(V_1\) by Schur's Lemma, we deduce that \(V_1\) occurs in the socle of \(V_p \tensor V_p\) with multiplicity \(1\).
Thus the missing summand is \(V_p\).
\end{proof}

\begin{example}\label{eg:tensor_of_simples}
Let \(p = 17\) and consider \(V_{14} \tensor V_9\).
We draw the \((14,9)\)-string below, and indicate those elements \(i\) for which \(2p-i \in \nmstring{14}{9}\) by joining \(i\) and \(2p-i\) with a dotted line.
The unpaired elements give rise to simple summands, while the paired elements give rise to projective indecomposable summands.
The summand of \(V_{14} \tensor V_9\) that arises out of each element of \(\nmstring{14}{9}\cap[17]\) is written below it.
\tikzset{
    dottedline/.style={thick, dotted},
    tick/.style={path picture={ 
        \draw[black]
        (path picture bounding box.south) -- (path picture bounding box.north);
    }},
    bigtick/.style={tick, circle, minimum size=0.5cm, fill=none, inner sep=0},
    smalltick/.style={tick, circle, minimum size=0.25cm, fill=none, inner sep=0},
    cross/.style={path picture={ 
        \draw[black]
        (path picture bounding box.south east) -- (path picture bounding box.north west) (path picture bounding box.south west) -- (path picture bounding box.north east);
    }},
    crossed/.style={cross, circle, minimum size=0.4cm, fill=none, inner sep=0},
}
\begin{center}
\begin{tikzpicture}[
    node distance=0.6cm,
]
    \node (6) [bigtick, label=below:\(6\)] {};
    \node (7) [smalltick, right of=6] {};
    \node (8) [bigtick, right of=7, label=below:\(8\)] {};
    \node (9) [smalltick, right of=8] {};
    \node (10) [bigtick, right of=9, label=below:\(10\)] {};
    \node (11) [smalltick, right of=10] {};
    \node (12) [bigtick, right of=11, label=below:\(12\)] {};
    \node (13) [smalltick, right of=12] {};
    \node (14) [bigtick, right of=13, label=below:\(14\)] {};
    \node (15) [smalltick, right of=14] {};
    \node (16) [bigtick, right of=15, label=below:\(16\)] {};
    \node (17) [crossed, right of=16, label=below:\(17\)] {};
    \node (18) [bigtick, right of=17, label=below:\(18\)] {};
    \node (19) [smalltick, right of=18] {};
    \node (20) [bigtick, right of=19, label=below:\(20\)] {};
    \node (21) [smalltick, right of=20] {};
    \node (22) [bigtick, right of=21, label=below:\(22\)] {};
    \draw (6.center) -- (22.center);
    
    \draw [bend left=90, dottedline] (12) to (22);
    \draw [bend left=90, dottedline] (14) to (20);
    \draw [bend left=90, dottedline] (16) to (18);
\begin{scope}[
    node distance=1.1cm,
]
    \node (V6) [below of=6] {\(V_6\)};
    \node (V8) [below of=8] {\(V_8\)};
    \node (V10) [below of=10] {\(V_{10}\)};
    \node (P12) [below of=12] {\(P_{12}\)};
    \node (P14) [below of=14] {\(P_{14}\)};
    \node (P16) [below of=16] {\(P_{16}\)};
    
    \node (o7) [below of=7] {\(\oplus\)};
    \node (o9) [below of=9] {\(\oplus\)};
    \node (o11) [below of=11] {\(\oplus\)};
    \node (o13) [below of=13] {\(\oplus\)};
    \node (o15) [below of=15] {\(\oplus\)};
\end{scope}

    \node (iso) [left of=V6, label=left:\(V_{14} \otimes V_9\)] {\(\cong\)};
\end{tikzpicture}
\end{center}

The pairing-up of \(i\) and \(2p-i\) corresponds to an isomorphism
\[
V_{2p-i} \iso
    \faktor{P_i}{V_i} \oplus \one{i=1}V_p.
\]
Indeed this isomorphism can be established by considering the short exact sequence
\[
\begin{tikzcd}
    0 \arrow[r] 
        & V_{p-1} \tensor V_{p-i} \arrow[r]
        & V_p \tensor V_{p-i+1} \arrow[r, "\mu"]
        & V_{2p-i} \arrow[r] 
        & 0.
\end{tikzcd}
\]
from \autoref{prop:mu_ses} and applying \autoref{thm:CG} to decompose \(V_p \tensor V_{p-i+1}\) and \(V_{p-1} \tensor V_{p-i}\).
\end{example}

\section{Tables of multiplicities}
\label{section:matrix}

We now suppose \(p > 2\).

We examine the table of multiplicities of simple modules as indecomposable summands of tensor products of simple modules, as well as the graph which has this table as its adjacency matrix.
This table has symmetries that reveal properties of the tensor products of representations of \(G\).
Furthermore, the Markov chain defined in the following section is shown to be a walk on this graph, so our observations here aid our understanding of that Markov chain.
We use \(\multiplicity{\ }{\ }\) to denote multiplicity as an indecomposable summand.

\begin{definition}
For \(n \in [p-1]\), let \(A^{(n)}\) be the matrix with entries \(A^{(n)}_{i,j} = \multiplicity{V_j}{V_i \tensor V_n}\).
Let \(\G^{(n)}\) be the (directed) graph (with loops) whose adjacency matrix is \(A^{(n)}\).
The parameter \(n\) is suppressed unless there is need to emphasise it.
\end{definition}

The matrix \(A\) is depicted in \autoref{fig:adj_matrix}.
It is visually apparent that \(A\) is symmetric; this motivates our next result.

\begin{figure}[ht]
\setlength{\padlen}{\widthof{\(\ddots\)}}
\renewcommand\arraystretch{1.1}
\[
\begin{blockarray}{*{16}{c}}
\begin{block}{*{16}{>{\scriptstyle}c}}
&  1     &  2   & \cdots &       &       &   n   &       &       &       &       &       &       &       &      &\\
\end{block}
\begin{block}{>{\scriptstyle}c(*{14}{>{\mathllap}c})>{\scriptstyle}c}
1&       & \pad  &       & \pad  &       &   1   &       &       &       &       &       &       &       &       & \\
2&       &       &       &       &   1   &       &   1   &       &       &       &       &       &       &       &\\
\vdots&  &       &       &   1   &       &   1   &       &   1   &       &       &       &       &       &       &\\
&        &       &\iddots&       &\iddots&       &\ddots &       &\ddots &       &       &       &       &       &\\
&        &   1   &       &   1   &       &   1   &       &   1   &       &   1   &       &       &       &       &\\
n&   1   &       &   1   &       &   1   &       &   1   &       &   1   &       &   1   &       &       &       &\\
&        &   1   &       &   1   &       &   1   &       &   1   &       &   1   &       &   1   &       &       &\\
&        &       &\ddots &       &\ddots &       &\ddots &       &\ddots &       &\ddots &       &\ddots &       &\\
&        &       &       &   1   &       &   1   &       &   1   &       &   1   &       &   1   &       &   1   & p-n\\
&        &       &       &       &   1   &       &   1   &       &   1   &       &   1   &       &   1   &       &\\
&        &       &       &       &       &\ddots &       &\ddots &       &\iddots&       &\iddots&       &       &\\
&        &       &       &       &       &       &   1   &       &   1   &       &   1   &       &       &       &\\
&        &       &       &       &       &       &       &   1   &       &   1   &       &       &       &       &\\
&        &       &       &       &       &       &       &       &   1   &       &       &       &       &       & p-1\\
\end{block}
\begin{block}{*{16}{>{\scriptstyle}c}}
&        &       &       &       &       &       &       &       &  \mathclap{\scriptstyle p-n}  &       &       &       &       & \mathclap{\scriptstyle p-1} & \\
\end{block}
\end{blockarray}
\]
\caption{The matrix \(A\) (here with \(n < p -n\))}
\label{fig:adj_matrix}
\end{figure}

\begin{lemma}\label{lemma:symmetry_of_division}
Suppose \(1 \leq i, j, l \leq p-1\).
The following are equivalent:
\begin{enumerate}[(i)]
    \item\label{VldividesViVj}
\simpledivide{l}{i}{j};
    \item\label{VidividesVjVl}
\simpledivide{i}{j}{l};
    \item\label{VjdividesVlVi}
\simpledivide{j}{l}{i};
    \item\label{symmetriccondition}
\(i+j+l \equiv 1 \pmod*{2}\), \(i+j+l < 2p\), and \(l < i + j\), \(i < j+l\) and \(j < l+i\).
\end{enumerate}
In particular, \(A\) is a symmetric matrix.
\end{lemma}

\begin{proof}[Proof\nopunct]
Observe that \ref{symmetriccondition} is symmetric in \(i\), \(j\) and \(l\), and so it suffices to show that \ref{VldividesViVj} and \ref{symmetriccondition} are equivalent.
Indeed, \autoref{thm:CG} tells us that \ref{VldividesViVj} holds if and only if \(l \equiv i+j-1 \pmod{2}\) and \(\max\set{i-j,j-i} < l < \min\set{i+j, 2p-(i+j)}\), which easily rearranges to \ref{symmetriccondition}.
\end{proof}

Thus \(\G\) can be viewed as an \emph{undirected} graph (with loops); we do so from now on.
Some small examples of \(\G\) are depicted in \autoref{fig:graphs}.

\begin{figure}[ht]
\tikzset{
    vertex/.style={circle,draw, minimum size=4.5mm, inner sep=0},
}
    \begin{subfigure}{0.49\linewidth}
        \centering
\begin{tikzpicture}[
    auto,
    node distance=1.1cm,
]
\begin{scope}[
    every node/.style={vertex},
]
    \node (1) [label=center:\(1\)] {};
    \node (2) [below of=1, label=center:\(2\)] {};
    \node (3) [right of=1, label=center:\(3\)] {};
    \node (4) [below of=3, label=center:\(4\)] {};
    \node (5) [right of=3, label=center:\(5\)] {};
    \node (6) [below of=5, label=center:\(6\)] {};
    \draw (1) -- (2) -- (3) -- (4) -- (5) -- (6);
\end{scope}
    \path[use as bounding box] (-0.9, 1) rectangle (3.1, -2);
\end{tikzpicture}
        \caption{\(n=2\)}
    \end{subfigure}
    \begin{subfigure}{0.49\linewidth}
        \centering
\begin{tikzpicture}[
    auto,
    node distance=1.1cm,
]
\begin{scope}[
    every node/.style={vertex},
]
    \node (1) [label=center:\(1\)] {};
    \node (2) [below of=1, label=center:\(2\)] {};
    \node (3) [right of=1, label=center:\(3\)] {};
    \node (4) [below of=3, label=center:\(4\)] {};
    \node (5) [right of=3, label=center:\(5\)] {};
    \node (6) [below of=5, label=center:\(6\)] {};
    \draw (1) -- (3) -- (5);
    \draw (3) to [out=55,in=125,looseness=9] (3);
    \draw (5) to [out=55,in=125,looseness=9] (5);
    \draw (2) -- (4) -- (6);
    \draw (2) to [out=-55,in=-125,looseness=9] (2);
    \draw (4) to [out=-55,in=-125,looseness=9] (4);
\end{scope}
    \path[use as bounding box] (-0.9, 1) rectangle (3.1, -2);
\end{tikzpicture}
        \caption{\(n=3\)}
    \end{subfigure}
    \begin{subfigure}{0.49\linewidth}
        \centering
\begin{tikzpicture}[
    auto,
    node distance=1.1cm,
]
\begin{scope}[
    every node/.style={vertex},
]
    \node (1) [label=center:\(1\)] {};
    \node (6) [below of=1, label=center:\(6\)] {};
    \node (3) [right of=1, label=center:\(3\)] {};
    \node (4) [below of=3, label=center:\(4\)] {};
    \node (5) [right of=3, label=center:\(5\)] {};
    \node (2) [below of=5, label=center:\(2\)] {};
    \draw (1) -- (4) -- (5) -- (2) -- (3) -- (6);
    \draw (3) -- (4);
\end{scope}
    \path[use as bounding box] (-0.9, 1) rectangle (3.1, -2);
\end{tikzpicture}
        \caption{\(n=4\)}
    \end{subfigure}
    \begin{subfigure}{0.49\linewidth}
        \centering
\begin{tikzpicture}[
    auto,
    node distance=1.1cm,
]
\begin{scope}[
    every node/.style={vertex},
]
    \node (1) [label=center:\(1\)] {};
    \node (6) [below of=1, label=center:\(6\)] {};
    \node (5) [right of=1, label=center:\(5\)] {};
    \node (3) [right of=5, label=center:\(3\)] {};
    \node (2) [below of=5, label=center:\(2\)] {};
    \node (4) [below of=3, label=center:\(4\)] {};
    \draw (1) -- (5) -- (3);
    \draw (3) to [out=-35,in=35,looseness=9] (3);
    \draw (6) -- (2) -- (4);
    \draw (4) to [out=-35,in=35,looseness=9] (4);
\end{scope}
    \path[use as bounding box] (-0.9, 1) rectangle (3.1, -2);
\end{tikzpicture}
        \caption{\(n=5\)}
    \end{subfigure}
    \caption{The graphs \(\G^{(n)}\), for \(p = 7\) and \(2 \leq n \leq p-2\)}
    \label{fig:graphs}
\end{figure}
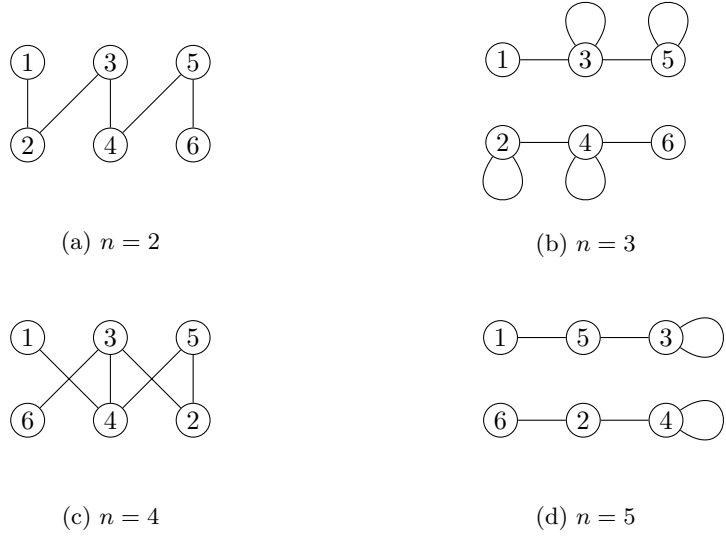

There is another visually apparent symmetry of the adjacency matrix \(A\): it is invariant under rotation by \(180\) degrees.
We give various interpretations of this fact in \autoref{prop:graph_symmetry}.
To give these interpretations, we make the following definitions.

\begin{definition}
Let \(T\) be the \((p-1) \times (p-1)\) matrix defined by \(T_{i,j} = [i+j = p]\).
\end{definition}

That is, \(T\) is the matrix with \(1\)s on the antidiagonal:
\renewcommand\arraystretch{1}
\[
    T =
\begin{pmatrix}
    &           & 1 \\
    & \raisebox{-2pt}{\ensuremath{\smash{\iddots}}}   &   \\
1   &           &   \\
\end{pmatrix}.
\]
\renewcommand\arraystretch{1}
It is the basis-change matrix for reversing the order of the basis, and is self-inverse.
Also:
\begin{itemize}
    \item
left-multiplying by \(T\) reflects a matrix in the horizontal midline;
    \item
right-multiplying by \(T\) reflects a matrix in the vertical midline;
    \item
conjugating by \(T\) rotates a matrix by \(180\) degrees.
\end{itemize}

\begin{definition}
Let \(\projfree{\blank}\) denote the projective-free part of a module.
\end{definition}

\begin{definition}
Recall that for an algebra \(\mathcal{A}\), the Grothendieck group \(G_0(\mathcal{A})\) is the abelian group with:
\begin{itemize}
    \item a generator \([V]\) for every \(\mathcal{A}\)-module \(V\), and
    \item a relation \([W] = [U] + [V]\) for every short exact sequence \(0 \to U \to W \to V \to 0\).
\end{itemize}
Let \(\p\) be the subgroup of the Grothendieck group \(G_0(\algclosedfield G)\) generated by classes of projective modules.
\end{definition}

Note that \(G_0(\algclosedfield G)\) can be made into a (commutative) ring via tensoring, and that \(\p\) is an ideal of this ring.
Recall that a quotient ring is naturally a (left) module for the original ring by (left) multiplication.

\begin{proposition}[Interpretations of rotational symmetry of \(A\)]
\label{prop:graph_symmetry}
The following statements hold:
\begin{enumerate}[(a)]
    \item\label{summand_implication}
\simpledivide{l}{i}{j} if and only if \simpledivide{l}{p-i}{p-j}, for all \(1 \leq i,j,l \leq p-1\);
    \item\label{TA}
\(A^{(n)} = TA^{(p-n)} = A^{(p-n)}T\);
    \item\label{TAT}
\(TAT = A\);
    \item\label{graph_automorpism}
the map \(i \mapsto p-i\) is a graph automorphism of \(\G\), and hence the induced subgraph on even vertices is isomorphic to the induced subgraph on odd vertices;
    \item\label{proj_free_iso}
\(\projfree{V_i \tensor V_j} \iso \projfree{V_{p-i} \tensor V_{p-j}}\) for all \(1 \leq i,j \leq p-1\);
    \item\label{equality_up_to_p}
\([V_i \tensor V_j] + \p = [V_{p-i} \tensor V_{p-j}] + \p\) for all \(1 \leq i, j \leq p-1\);
    \item\label{Grothendieck_group_mod_iso}
the \(\algclosedfield\)-linear automorphism \(\rho\) of \(\origfaktor{G_0(\algclosedfield G)}{\p}\) defined by \(\rho \colon [V_i] + \p \mapsto [V_{p-i}] + \p\) is \(G_0(\algclosedfield G)\)-equivariant.
\end{enumerate}
\end{proposition}

\begin{proof}[Proof\nopunct]
Statement \ref{summand_implication} and the first equality in \ref{TA} are equivalent, and the second equality in \ref{TA} follows from the first since \(A\) and \(T\) are symmetric.
The statements \ref{TAT} and \ref{graph_automorpism} are equivalent, and are implied by \ref{TA}.
% The second part of \ref{graph_automorpism} is clear from the first by observing that the map \(i \mapsto p-i\) swaps the parity of each vertex.
The statement \ref{proj_free_iso} clearly implies \ref{equality_up_to_p}, and given that the projective-free parts of the tensor products of simple modules are multiplicity-free sums of simple modules, both are equivalent to \ref{summand_implication}.

To see that \ref{Grothendieck_group_mod_iso} follows from \ref{TA}, let \(\Lambda \subseteq [p-1]\) be such that \(\Omega^0(V_j \tensor V_i) \iso \bigoplus_{l \in \Lambda} V_l\).
Then, by the second equality in \ref{TA}, we have \(\Omega^0(V_{j} \tensor V_{p-i}) \iso \bigoplus_{l \in \Lambda} V_{p-l}\); thus
\begin{align*}
\rho([V_j \tensor V_i] + \p)
    &= \rho\left(\sum_{l \in \Lambda} [V_{l}] + \p\right) \\
    &= \sum_{l \in \Lambda} [V_{p-l}] + \p \\
    &= [V_j \tensor V_{p-i}] + \p,
\end{align*}
as required.

Thus it suffices to show \ref{summand_implication} holds.
Indeed, condition \ref{symmetriccondition} of \autoref{lemma:symmetry_of_division} is invariant under taking both \(i \mapsto p-i\) and \(j \mapsto p-j\).
\end{proof}

\begin{remark}
The observations of \autoref{lemma:symmetry_of_division} and \autoref{prop:graph_symmetry} can be seen as observations about the fusion category corresponding to the algebraic group \(\SL_2(\algclosedfield)\) described in \cite[Section 2]{andersen1995}.
The tensor product in this category is ``reduced'', in a sense which in our setting means ``modulo projectives''.
The quotient ring \(\origfaktor{G_0(\algclosedfield)}{\p}\) is known as the fusion ring for this category.
The fusion rules state how reduced tensor products decompose in this category, and thus are specified by our \autoref{lemma:symmetry_of_division}.
The observed symmetries of these fusion rules can be deduced either axiomatically \cite[Axiom 3, p.~183]{mathieu2000} or as a consequence of the modular Verlinde formula \cite[Theorem 9.5]{mathieu2000}.
\end{remark}

We next observe that a certain submatrix of \(A\) contains all the information of \(A\), and use the resulting simplification of the structure of \(A\) to identify the connected components of \(\G\).

\begin{definition}
Let \(\bar{A}^{(n)}\) be the \(\frac{p-1}{2} \times \frac{p-1}{2}\) submatrix of (a conjugate of) \(A\) defined by
\begin{align*}
    \bar{A}^{(n)}_{i,j} &= \begin{cases}
        A^{(n)}_{2i-1,2j-1} &\text{if \(n\) is odd;} \\
        A^{(n)}_{2i-1,p+1-2j} &\text{if \(n\) is even.}
    \end{cases}
\end{align*}
\end{definition}
That is, if the vertices are reordered to \(1, 3, \ldots, p-2, p-1, p-3, \ldots, 4, 2\) (the odd integers followed by the even integers, with the former in ascending order and the latter in descending order), then \(\bar{A}\) is the upper-left block of \(A\) when \(n\) is odd and is the upper-right block of \(A\) when \(n\) is even.

\begin{lemma}\label{lemma:A_bar_props}
The matrix \(\bar{A}\) has the following properties:
\begin{enumerate}[(a)]
    \item\label{block_form_of_A}
under the ordering \(1, 3 \ldots p-2, p-1, \ldots, 4,2\), the matrix \(A\) is of the form
\renewcommand{\star}{\ast}
\[
    A = \begin{dcases}
    \begin{pmatrix}
            \bar{A} & \star      \\
            \star   & \bar{A} \\
        \end{pmatrix} &\text{if \(n\) is odd,} \\
        \begin{pmatrix}
            \star   & \bar{A} \\
            \bar{A} & \star      \\
        \end{pmatrix} &\text{if \(n\) is even,}
    \end{dcases}
\]
where \(\star\) denotes an unspecified matrix;
    \item\label{nonzero_condition}
\(\bar{A}^{(n)}_{i,j} = 1\) if and only if \(2\abs{i-j} < r < 2(i+j-1) < 2p-r\), where \(r = n\) if \(n\) is odd and \(r = p-n\) if \(n\) is even.
    \item\label{A_bar_p-n}
\(\bar{A}^{(p-n)} = \bar{A}^{(n)}\);
    \item\label{A_bar_symmetric}
\(\bar{A}\) is symmetric;
    \item\label{A_bar_regular}
for \(1 < n < p-1\), the graph with adjacency matrix \(\bar{A}\) is connected.
\end{enumerate}
\end{lemma}

\begin{proof}[Proof\nopunct]
By \autoref{prop:graph_symmetry}\ref{TAT} we have \(A_{2i-1,2j-1} = A_{p+1-2i,p+1-2j}\), and so (under the new ordering) the upper-left and lower-right blocks of \(A\) are the same. Similarly the upper-right and lower-left blocks are the same, and \ref{block_form_of_A} follows.

The condition for \(\bar{A}_{i,j}\) to be nonzero is obtained from condition \ref{symmetriccondition} of \autoref{lemma:symmetry_of_division} with the appropriate values of \(i\) and \(j\) substituted.
Properties \ref{A_bar_p-n} and \ref{A_bar_symmetric} are easily verified using this condition.

It follows from \ref{nonzero_condition} that \(\bar{A}\) has nonzero entries precisely in a rectangle bounded by the straight lines determined by these inequalities; we draw matrix \(\bar{A}\) in \autoref{fig:A_bar}.
The connectedness of its graph is then clear provided \(r \neq 1\).
\end{proof}

\begin{figure}
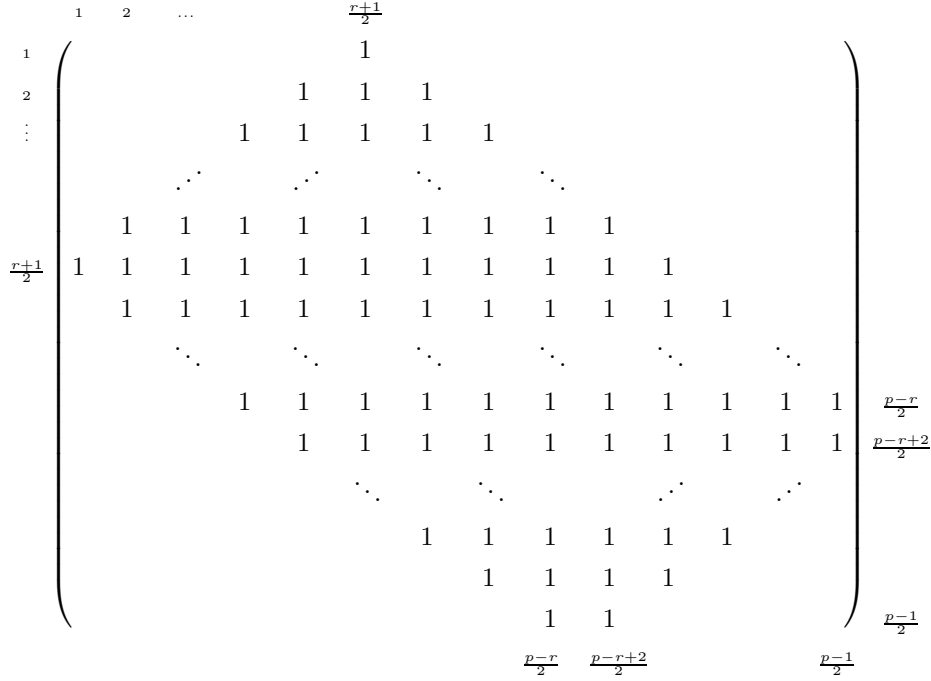

\setlength{\padlen}{\widthof{\(\ddots\)}}
\setlength{\padlen}{0.965\padlen}
\renewcommand\arraystretch{1.3}
\[
\begin{blockarray}{*{16}{c}}
\begin{block}{*{16}{>{\scriptscriptstyle}c}}
&  1   &   2   & \ldots &      &       &   \mathclap{\frac{r+1}{2}}   &       &       &       &       &       &       &       &       & \\
\end{block}
\begin{block}{>{\scriptscriptstyle}c(*{14}{>{\mathllap}c})>{\scriptscriptstyle}c}
1&     & \pad{}&       & \pad{}&       &   1   &       &       &       & \pad{}&       & \pad{}&       &       & \\
2&     &       &       &       &   1   &   1   &   1   &       &       &       &       &       &       &       & \\
\vdots&&       &       &   1   &   1   &   1   &   1   &   1   &       &       &       &       &       &       & \\
&      &       &\iddots&       &\iddots&       &\ddots &       &\ddots &       &       &       &       &       & \\
&      &   1   &   1   &   1   &   1   &   1   &   1   &   1   &   1   &   1   &       &       &       &       & \\
\frac{r+1}{2}&1& 1 & 1 &   1   &   1   &   1   &   1   &   1   &   1   &   1   &   1   &       &       &       & \\
&      &   1   &   1   &   1   &   1   &   1   &   1   &   1   &   1   &   1   &   1   &   1   &       &       & \\
&      &       &\ddots &       &\ddots &       &\ddots &       &\ddots &       &\ddots &       &\ddots &       & \\
&      &       &       &   1   &   1   &   1   &   1   &   1   &   1   &   1   &   1   &   1   &   1   &   1   & \frac{p-r}{2} \\
&      &       &       &       &   1   &   1   &   1   &   1   &   1   &   1   &   1   &   1   &   1   &   1   & \frac{p-r+2}{2} \\
&      &       &       &       &       &\ddots &       &\ddots &       &       &\iddots&       &\iddots&       & \\
&      &       &       &       &       &       &   1   &   1   &   1   &   1   &   1   &   1   &       &       & \\
&      &       &       &       &       &       &       &   1   &   1   &   1   &   1   &       &       &       & \\
&      &       &       &       &       &       &       &       &   1   &   1   &       &       &       &       & \frac{p-1}{2} \\
\end{block}
\begin{block}{*{16}{>{\scriptscriptstyle}c}}
&      &       &       &       &       &&&&\mathclap{\frac{p-r}{2}\phantom{p1}}&\mathclap{\phantom{p1}\frac{p-r+2}{2}}&&&&\mathclap{\frac{p-1}{2}}&\\
\end{block}
\end{blockarray}
\]
    \caption{The matrix \(\bar{A}\) (here with \(r < p - r\)), where \(r = n\) if \(n\) is odd and \(r = p-n\) if \(n\) is even}
    \label{fig:A_bar}
\end{figure}

\begin{lemma}\label{prop:parity_observations}{\ }
\begin{enumerate}[(a)]
    \item\label{odds_disconnected}
If \(n\) is odd, then \(\G\) is disconnected, with each connected component a subset of either the odd integers or the even integers.
    \item\label{evens_bipartite}
If \(n\) is even, then \(\G\) is bipartite, with classes the odd integers and the even integers.
    \item\label{tensor_form_of_A}
When the vertices are ordered as \(1, 3, \ldots, p-2, p-1, p-3, \ldots, 4, 2\), we have
\[
    A = \bar{A} \tensor \begin{pmatrix}
    0 & 1 \\
    1 & 0 \\
    \end{pmatrix}^{n+1}.
\]
\end{enumerate}
\end{lemma}

\begin{proof}[Proof\nopunct]
Let \(1 \leq i \leq p-1\). Observe that the neighbours of \(i\) are all elements of \(\nmstring{i}{n}\) or \(\nmstring{n}{i}\) (according to whether \(i \geq n\) or \(i \leq n\)).
Furthermore, elements of these strings are all of the same parity, which is the parity of \(i+n-1\).
Thus if \(n\) is odd, the neighbours of \(i\) are of the same parity as \(i\), whilst if \(n\) is even, the neighbours of \(i\) are of different parity to \(i\).
The statements \ref{odds_disconnected} and \ref{evens_bipartite} are then immediate.

That is, under the new ordering, when \(n\) is even the diagonal blocks of \(A\) are zero, and when \(n\) is odd the off-diagonal blocks are zero. The expression as a Kronecker product then follows from \autoref{lemma:A_bar_props}\ref{block_form_of_A}.
\end{proof}

\begin{proposition}\label{lemma:connected_components}{\ }
\begin{enumerate}[(a)]
    \item\label{odds_two_components}
If \(n\) is odd and \(n > 1\), then \(\G\) has precisely two connected components, the odd integers and the even integers, and they are isomorphic.
    \item\label{evens_connected}
If \(n\) is even and \(n < p-1\), then \(\G\) is connected.
\end{enumerate}
\end{proposition}

\begin{proof}[Proof\nopunct]
For \(n\) odd, \(\bar{A}\) is the adjacency matrix for the subgraphs of \(\G\) on odd vertices and on even vertices, so \ref{odds_two_components} follows immediately from \autoref{lemma:A_bar_props}\ref{A_bar_regular}.

For \(n\) even, \(\bar{A}\) is the adjacency matrix for the quotient graph of \(\G\) with \(i\) and \(p-i\) identified. Again using \autoref{lemma:A_bar_props}\ref{A_bar_regular}, since \(\G\) is bipartite (with each of \(i\) and \(p-i\) in a distinct class), to show \ref{evens_connected} it suffices to show that \(i\) is reachable from \(p-i\) for some \(i\). Indeed, \(\bar{A}\) has a nonzero diagonal entry (at \(\frac{n+1}{2}\)), and so the two vertices identified to form the corresponding vertex of the quotient are adjacent.
\end{proof}

We conclude this section by finding the degrees of the vertices in \(\G\).
The degree of \(i\) in \(\G\) is also the number of nonzero entries in the \(i\)th row of \(A\), and is the number of non-projective indecomposable summands of \(V_i \tensor V_n\).

\begin{definition}\label{def:degree}
For \(1 \leq i \leq p-1\), let \(d(i)\) be the degree of \(i\) in \(\G\) (where a loop is considered to contribute \(1\) to the degree).
The dependence of \(d\) on \(n\) is suppressed, since it is always clear from context.
\end{definition}

\begin{lemma}\label{lemma:d(i)values}
For \(1 \leq i \leq p-1\), we have
\[
d(i) = \min\set{i, p-i, n, p-n}.
\]
Furthermore,
\[
    \sum_{i=1}^{p-1} d(i) = n(p-n).
\]
\end{lemma}

\begin{proof}[Proof\nopunct]
Clearly \(d(i)\) is symmetric in \(i\) and \(n\), so for the first equality it suffices to show that \(d(i) = \min\set{i,p-n}\) when \(i \leq n\).
By \autoref{thm:CG}, the number of simple non-projective summands of \(V_n \tensor V_i\) is the number of elements \(j\) of \(\nmstring{n}{i}\) for which \(2p-j \notin \nmstring{n}{i}\).

If \(i + n -1 < p\) (equivalently, \(i \leq p-n\)) then this is all the elements of \(\nmstring{n}{i}\), of which there are \(i\).

If \(i + n - 1 \geq p\) (equivalently, \(i > p-n\)), then the number of \(j \in \nmstring{n}{i}\) such that \(2p-j \in \nmstring{n}{i}\) is
\[
    2 \floor{\frac{i+n-1-p}{2}} + [\text{\(i+n-1\) is odd}]
    = i+n- p,
\]
and so \(d(i) = i - (i+n-p) = p-n\).

We now find the sum of the \(d(i)\).
Let \(m = \min\set{n, p-n}\).
We have:
\begin{align*}
    \sum_{i=1}^{p-1} d(i)
        &= \sum_{i=1}^{p-1} \min\set{i, p-i, n, p-n} \\
        &= 2\sum_{i=1}^{\frac{p-1}{2}} \min\set{i, m} \\
        &= 2\left(\frac{p-1}{2} - m\right)m
            + 2\sum_{i=1}^{m} i \\
        &= m(p-1-2m) + m(m+1) \\
        &= m(p-m) \\
        &= n(p-n)
\end{align*}
as required.
\end{proof}

\section{Random walks on indecomposable modules}
\label{section:random_walks}

%Let \(G = \SL_2(\primesubfield)\) and \(p \neq 2\) throughout this section.

We investigate the long-run behaviour of tensoring by a fixed simple \(\algclosedfield G\)-module by considering the Markov chain defined below.
In particular, we assess the properties of reversibility, diagonalisability, irreducibility and periodicity, as well as calculating stationary distributions and mixing times.

\begin{definition}[Non-projective summand random walk]
Fix \(n \in [p-1]\), \(w\) a function that assigns a positive weight to each non-projective indecomposable \(\algclosedfield G\)-module, and \(\nu\) a distribution on the non-projective simple \(\algclosedfield G\)-modules.
Let the \emph{non-projective summand random walk} be the (discrete time) Markov chain on the set of non-projective indecomposable \(\algclosedfield G\)-modules
with initial distribution \(\nu\) in which the probability of a step from \(U\) to \(V\) is
\[
    Q^{(n)}_{UV} = \frac
    {w(V) \multiplicity{V}{U \tensor V_n}}
    {\sum_{M} w(M) \multiplicity{M}{U \tensor V_n}},
\]
where the sum is over all non-projective indecomposable modules \(M\) (and \(\multiplicity{\ }{\ }\) denotes multiplicity as an indecomposable summand, as in \autoref{section:matrix}).
The parameter \(n\) is suppressed unless there is need to emphasise it.
\end{definition}

\begin{remark}{\ }
\begin{enumerate}[(i)]
    \item 
If \(U\) is a simple non-projective \(\algclosedfield G\)-module, \autoref{thm:CG} implies that \(U \tensor V_n\) indeed has non-projective indecomposable summands, and that these summands are simple. %, and that they occur with multiplicity at most \(1\).
Thus the chain is well-defined and remains on simple non-projective \(\algclosedfield G\)-modules throughout.
The states of the chain can therefore be labelled with the dimensions of the modules, taking values in the finite set \([p-1]\).

    \item
\autoref{thm:CG} also implies that the non-projective part of a tensor product of simple modules is multiplicity-free, so \(\multiplicity{M}{U \tensor V_n} \in \set{0,1}\) for all \(M\).

    \item
If we were to allow steps to projective indecomposable modules, these modules would form an absorbing set (in the sense that once the chain hit a projective module it would stay on projective modules for all time).
This definition allows us to consider a recurrent chain on the (non-projective) simple modules.

    \item
There are two trivial cases to be excluded: if \(n=1\), we never step away from the initial state;
if \(n=p-1\), then \(V_{p-i}\) is the unique non-projective indecomposable summand of \(V_i \tensor V_{p-1}\), so at each step we switch between the initial state \(i\) and \(p-i\). From now on we assume \(2 \leq n \leq p-2\).
\end{enumerate}
\end{remark}

An illustrative example of our chain is given below.
Note that when \(w \equiv 1\), the summands are chosen uniformly at random;
this case, and the case where \(w(i)=i\) (in which modules are weighted by their dimension), are described for general \(n\) at the end of this section.

\begin{example}
Suppose \(w \equiv 1\) and \(n=2\).
We have that
\[
    V_i \tensor V_2 \iso \begin{dcases}
        V_2 &\text{if \(i=1\),} \\
        V_{i-1} \oplus V_{i+1} &\text{if \(2 \leq i \leq p-2\),} \\ 
        V_{p-2} \oplus P_p &\text{if \(i=p-1\).}
    \end{dcases}
\]
Thus the non-projective summand random walk is a symmetric random walk in one dimension with reflecting boundaries. The transition matrix is
\[
\begin{pmatrix}
        & 1       &         &         &         &         &         &         \\
\frac12 &         & \frac12 &         &         &         &         &         \\
        & \frac12 &         & \frac12 &         &         &         &         \\
        &         & \frac12 &         & \frac12 &         &         &         \\
        &         &         & \ddots  &         & \ddots  &         &         \\
        &         &         &         & \frac12 &         & \frac12 &         \\
        &         &         &         &         & \frac12 &         & \frac12 \\
        &         &         &         &         &         & 1       &         \\
\end{pmatrix}
\]
and the stationary distribution is \(\frac{1}{2(p-2)}(1,2,2,\ldots,2,1)\).
\end{example} 

Our key observation while studying the non-projective summand random walk is that it is the random walk on the graph \(\G\) (defined in \autoref{section:matrix}) in which the probability of moving from a vertex \(i\) to a neighbour \(j\) is proportional to \(w(j)\).
Indeed, the transition matrix \(Q\) has nonzero entries precisely where \(A\) (the adjacency matrix for \(\G\)) does, and in both cases the transition probabilities are proportional to the weight of the destination.
That is,
\[
    Q_{i,j} = \frac{w(j)}{\sum_{il \in E(\G)}w(l)} A_{i,j}.
\]

We use the properties of \(\G\) given in \autoref{section:matrix} to shed light on the non-projective summand random walk.
The first relevant property of \(\G\) is that it is undirected, which implies that the communicating classes of our Markov chain are all closed (that is, they are irreducible chains themselves) and they are precisely the connected components of \(\G\).
Moreover, by the following lemma, it implies the chain is reversible and diagonalisable, and we are able to find a stationary distribution.

\begin{lemma}\label{lemma:walk_on_a_graph}
Let \(\Hscr\) be any finite graph (with loops) and \(u\) a function assigning a positive weight to each vertex of \(\Hscr\).
Let \(R\) be the transition matrix for the random walk on \(\Hscr\) defined by
\[
    R_{i,j} = \frac{u(j)}{\sum_{il \in E(\Hscr)} u(l)} \one{ij \in E(\Hscr)}.
\]
Let \(\pi\) be the distribution defined by
\[
    \pi_i = \frac{u(i) \sum_{il \in E(\Hscr)} u(l)}{C}.
\]
where \(C = \sum_{x \in V(\Hscr)}\sum_{xy \in E(\Hscr)} u(y)\).

Then \(\pi\) is a stationary 
distribution in detailed balance with \(R\), and the random walk is reversible and diagonalisable.
\end{lemma}

\begin{proof}[Proof\nopunct]
It suffices to verify the detailed balance equations for \(\pi\) (noting that diagonalisability follows from reversibility \cite[Section 2.4]{Pistone2013}).
Observe:
\begin{align*}
\pi_i R_{i,j}
    &=  \frac{u(j)}{\sum_{il \in E(\Hscr)} u(l)} \frac{u(i) \sum_{il \in E(\Hscr)} u(l)}{C} \one{ij \in E(\Hscr)} \\
    &= \frac{u(i)u(j)}{C} \one{ij \in E(\Hscr)} \\
    &= \pi_j R_{j,i},
\end{align*}
as required.
\end{proof}

Next, we make use of our results about the connectedness and periodicity of \(\G\).

\begin{proposition}\label{prop:chain_classification}{\ }
\begin{enumerate}[(a)]
    \item\label{odds_classification}
If \(n\) is odd, %and \(n > 1\),
then the non-projective summand random walk is reducible into two chains, one on the even states and one on the odd states, which are each irreducible and aperiodic.
    \item\label{evens_classification}
If \(n\) is even, %and \(n < p-1\),
then the non-projective summand random walk is irreducible and periodic with period \(2\).
\end{enumerate}
\end{proposition}

\begin{proof}[Proof\nopunct]
The description of the irreducible components follows immediately from the description of the connected components of \(\G\) in \autoref{lemma:connected_components}.

A walk on an undirected graph necessarily has period at most \(2\) (since any vertex can be revisited after two steps). The walk has period equal to \(2\) if and only if the graph contains no odd cycles and no loops, which is if and only if the graph is bipartite---and the walk is aperiodic otherwise.
Thus the periodicity claims follow from \autoref{prop:parity_observations}\ref{evens_bipartite} and the observation that when \(n\) is odd, each component of \(\G\) has loops (at \(\frac{p-1}{2}\) and \(\frac{p+1}{2}\)).
\end{proof}

\begin{remark}
Thus for \(n\) even,
the chain has a \emph{unique} stationary distribution but it does not necessarily converge to it.
Meanwhile, for \(n\) odd,
each subchain has a unique stationary distribution which it converges to, and the stationary distributions of the entire chain are precisely the convex combinations of these distributions.
\end{remark}

If \(w\) satisfies \(w(i)=w(p-i)\) for all \(i\), then \(Q\) has the same rotational symmetry as \(A\), and several of the results from \autoref{section:matrix} carry over.
Some of these results are helpful for identifying the remaining eigenvalues of \(Q\); the rate of convergence to equilibrium is determined by the second-largest (in absolute value) eigenvalue, so this in turn is helpful for finding the mixing time for the Markov chain.

Let \(\bar{Q}\) be the submatrix of (a conjugate of) \(Q\) defined analogously to \(\bar{A}\).

\begin{proposition}\label{prop:w(i)=w(p-i)_props}
Suppose \(w(i) = w(p-i)\) for all \(i\). Then:
\begin{enumerate}[(a)]
    \item\label{QT}
\(Q^{(n)} = TQ^{(p-n)} = Q^{(p-n)}T\);
    \item\label{TQT}
\(TQT = Q\);
    \item
the non-projective summand random walk is invariant under the relabelling \(i \mapsto p-i\);
    \item
if \(n\) is odd, the two irreducible subchains are isomorphic;
    \item
\(\bar{Q}^{(p-n)} = \bar{Q}^{(n)}\);
    \item\label{Kronecker_Q}
with the vertices are ordered as \(1, 3, \ldots, p-2, p-1, p-3, \ldots, 4, 2\), we have
\[
    Q = \bar{Q} \tensor \begin{pmatrix}
    0 & 1 \\
    1 & 0 \\
    \end{pmatrix}^{n+1};
\]
    \item\label{Q_eigenvalues}
if \(n\) is odd, every eigenvalue of \(Q\) has even multiplicity; if \(n\) is even, the eigenvalues of \(Q\) come in signed pairs;
    \item\label{mixing-time}
the non-projective summand random walk has mixing time bounded by
\[
    t_{\textnormal{mix}}(\epsilon) \leq \frac{1}{1-\lambda_\star} \log\left(\frac{1}{\epsilon \min_i(\pi_i)}\right)
\]
where \(\lambda_\star = \max\setbuild{\abs{\lambda}}{\text{\(\lambda \neq 1\) is an eigenvalue of \(\bar{Q}\)}}\).
\end{enumerate}
\end{proposition}

\begin{proof}[Proof\nopunct]
Statements \ref{QT}--\ref{Kronecker_Q} are entirely analogous to results in \autoref{section:matrix}, using \(w(i)=w(p-i)\) to deduce that the entries in the desired places of \(Q\) are not only nonzero but also equal.

Once we have the Kronecker product expression in \ref{Kronecker_Q}, we see immediately that if \(\bar{Q}\) has eigenvector-eigenvalue pairs \(\set{(v_1, \lambda_1), \ldots, (v_{\frac{p-1}{2}}, \lambda_{\frac{p-1}{2}})}\), then \(Q\) has eigenvector-eigenvalue pairs
\begin{alignat*}{2}
\setbuild{\left(v_i \tensor \begin{psmallmatrix}
    1 \\ 0 
    \end{psmallmatrix}, \lambda_i\right)}{1 \leq i \leq \tfrac{p-1}{2}}
        &\sqcup
    \setbuild{\left(v_i \tensor \begin{psmallmatrix}
    0 \\ 1 
    \end{psmallmatrix}, \lambda_i\right)}{1 \leq i \leq \tfrac{p-1}{2}} &\quad&\text{if \(n\) is odd;} \\
\setbuild{\left(v_i \tensor \begin{psmallmatrix}
    1 \\ 1 
    \end{psmallmatrix}, \lambda_i\right)}{1 \leq i \leq \tfrac{p-1}{2}}
        &\sqcup
    \setbuild{\left(v_i \tensor \begin{psmallmatrix}
    \phantom{-}1 \\ -1 
    \end{psmallmatrix}, -\lambda_i\right)}{1 \leq i \leq \tfrac{p-1}{2}} &\quad&\text{if \(n\) is even.}
\end{alignat*}
Both parts of \ref{Q_eigenvalues} then follow.

Note that \(\bar{Q}\) is the transition matrix for an irreducible aperiodic chain, so all its eigenvalues lie in \((-1,1]\) and the eigenvalue \(1\) has multiplicity \(1\); therefore \(\lambda_\star < 1\) and \(\lambda_\star\) is the second-largest eigenvalue (in absolute value) of \(\bar{Q}\).

If \(n\) is odd, \(\lambda_\star\) is therefore the second-largest eigenvalue (in absolute value) for each irreducible component of the chain.
If \(n\) is even, in order to eliminate periodicity, we define the lazy chain with transition matrix \(\frac12 (Q+I)\) (which converges at half the rate of the original chain); since the eigenvalues of \(\bar{Q}\) come in signed pairs, the lazy chain has second-largest eigenvalue \(\frac{\lambda_\star +1}{2}\).
Then the bound on the mixing time follows from \cite[Theorem 12.4, p.~163]{levin2017markov} (halving the mixing time of the lazy chain when \(n\) is even).
\end{proof}

In fact, for \(n\) even, the eigenvalues still come in signed pairs, regardless of the weighting: it is always the case that \(Q\) has nonzero entries only in the off-diagonal \(\frac{p-1}{2} \times \frac{p-1}{2}\) blocks, and if \(\begin{psmallmatrix}
u \\ v
\end{psmallmatrix}\) is an eigenvector with eigenvalue \(\lambda\) for such a matrix, then \(\begin{psmallmatrix}
\phantom{-}u \\ -v
\end{psmallmatrix}\) is an eigenvector with eigenvalue \(-\lambda\).
%However, in general, there is not a simple relation between these off-diagonal blocks, or to the blocks of the transition matrix with \(p-n\) in place of \(n\).

We conclude by exhibiting our results in the cases \(w \equiv 1\) and \(w(i)=i\).
Recall from \autoref{def:degree} that \(d(i)\) is the degree of \(i\) in \(\G\).

\begin{example}\label{eg:w_equiv_1}
Let \(w \equiv 1\). Then
\[
    Q_{i,j} = \frac{A_{i,j}}{d(i)}.
\]

This transition matrix is shown explicitly in \autoref{fig:trans_mats}. Of course, \(w(i) = w(p-i)\), and so \(Q\) satisfies \(TQT = Q\), and for \(n\) odd the the two irreducible subchains are isomorphic.

\begin{figure}
    \centering
\setlength{\padlen}{\widthof{\(\scriptstyle\frac{1}{p-n-1}\)}}
\setlength{\edgepadlen}{\widthof{\(\scriptstyle\frac{1}{p-n}\)}}
\renewcommand\arraystretch{1.3}
\makeatletter
\renewcommand\BA@colsep{0pt}
\makeatother
\newlength{\boundarylen}
\setlength{\boundarylen}{12pt}
\[
\begin{blockarray}{*{16}{c}}
\begin{block}{*{16}{>{\scriptstyle}c}}
&  1     &  2   & \cdots &       &       &   n   &       &       &       &       &       &       &       &      &\\
\end{block}
\begin{block}{>{\scriptstyle}c<{\hspace{\boundarylen}}(*{14}{>{\mathllap}c})>{\hspace{\boundarylen}\scriptstyle}c}
1&\edgepad&\pad&\edgepad&\pad  & \pad  &\pad[1]&  \pad & \pad  & \pad  &  \pad & \pad  &\edgepad& \pad &\edgepad&\\
2&     &       &       &       &\frac12&       &\frac12&       &       &       &       &       &       &        &\\
\vdots&&       &       &\frac13&       &\frac13&       &\frac13&       &       &       &       &       &        &\\
&      &       &\iddots&       &\iddots&       &\ddots &       &\ddots &       &       &       &       &        &\\
&      &\frac{1}{n-1}& &\frac{1}{n-1}& &\frac{1}{n-1}& &\frac{1}{n-1}& &\frac{1}{n-1}& &       &       &        &\\
n&\frac1n&     &\frac1n&       &\frac1n&       &\frac1n&       &\frac1n&       &\frac1n&       &       &        &\\
&      &\frac1n&       &\frac1n&       &\frac1n&       &\frac1n&       &\frac1n&       &\frac1n&       &        &\\
&      &       &\ddots &       &\ddots &       &\ddots &       &\ddots &       &\ddots &       &\ddots &        &\\
&      &       &       &\frac1n&       &\frac1n&       &\frac1n&       &\frac1n&       &\frac1n&       &\frac1n & p-n\\
&      &       &       &       &\frac{1}{n-1}& &\frac{1}{n-1}& &\frac{1}{n-1}& &\frac{1}{n-1}& &\frac{1}{n-1}&  &\\
&      &       &       &       &       &\ddots &       &\ddots &       &\iddots&       &\iddots&       &        &\\
&      &       &       &       &       &       &\frac13&       &\frac13&       &\frac13&       &       &        &\\
&      &       &       &       &       &       &       &\frac12&       &\frac12&       &       &       &        &\\
&      &       &       &       &       &       &       &       &   1   &       &       &       &       &        & p-1\\
\end{block}
\begin{block}{*{16}{>{\scriptstyle}c}}
&      &       &       &       &       &       &      &&\mathclap{\scriptstyle p-n}&& &  &&\mathclap{\scriptstyle p-1}&\\
\end{block}
\end{blockarray}
\]

\[
\begin{blockarray}{*{16}{c}}
\begin{block}{*{16}{>{\scriptstyle}c}}
&  1   &  2    & \cdots&       &       &       &       &       &   n   &       &       &       &       &      &\\
\end{block}
\begin{block}{>{\scriptstyle}c<{\hspace{\boundarylen}}(*{14}{>{\mathllap}c})>{\hspace{\boundarylen}\scriptstyle}c}
1&     &\pad   &       & \pad  & \pad  &  \pad & \pad  & \pad  &\pad[1]&  \pad & \pad  &       & \pad  &        &\\
2&     &       &       &       &       &       &       &\frac12&       &\frac12&       &       &       &        &\\
\vdots&&       &       &       &       &       &\frac13&       &\frac13&       &\frac13&       &       &        &\\
&      &       &       &       &       &\iddots&       &\iddots&       &\ddots &       &\ddots &       &        &\\
&      &       &       & &\scriptstyle\frac{1}{p-n-1}& &\scriptstyle\frac{1}{p-n-1}& &\scriptstyle\frac{1}{p-n-1}& &\scriptstyle\frac{1}{p-n-1}& &\scriptstyle\frac{1}{p-n-1}&  &\\
&      &   & &\scriptstyle\frac{1}{p-n}& &\scriptstyle\frac{1}{p-n}& &\scriptstyle\frac{1}{p-n}& &\scriptstyle\frac{1}{p-n}& &\scriptstyle\frac{1}{p-n}& &\scriptstyle\frac{1}{p-n} & p-n\\  
&      & &\scriptstyle\frac{1}{p-n}& &\scriptstyle\frac{1}{p-n}& &\scriptstyle\frac{1}{p-n}& &\scriptstyle\frac{1}{p-n}& &\scriptstyle\frac{1}{p-n}& &\scriptstyle\frac{1}{p-n}&  &\\
&      &\iddots&       &\iddots&       &\iddots&       &\iddots&       &\iddots&       &\iddots&       &        &\\      
n&\scriptstyle\frac{1}{p-n}& &\scriptstyle\frac{1}{p-n}& &\scriptstyle\frac{1}{p-n}& &\scriptstyle\frac{1}{p-n}& &\scriptstyle\frac{1}{p-n}& &\scriptstyle\frac{1}{p-n}& &   &  & &\\
&      &\scriptstyle\frac{1}{p-n-1}& &\scriptstyle\frac{1}{p-n-1}& &\scriptstyle\frac{1}{p-n-1}& &\scriptstyle\frac{1}{p-n-1}& &\scriptstyle\frac{1}{p-n-1}& &       &       &       &\\   
&      &       &\ddots&       &\ddots&       &\iddots &       &\iddots &       &       &       &       &       &\\
&      &       &       &\frac13&       &\frac13&       &\frac13&       &       &       &       &       &       &\\
&      &       &       &       &\frac12&       &\frac12&       &       &       &       &       &       &       &\\
&      &       &       &       &       &   1   &       &       &       &       &       &       &       &       &p-1\\
\end{block}
\begin{block}{*{16}{>{\scriptstyle}c}}
&      &       &       & &&\mathclap{\scriptstyle p-n}&& &       &       &       & &&\mathclap{\scriptstyle p-1}&\\
\end{block}
\end{blockarray}
\]
    \caption{The transition matrix \(Q\) when \(w \equiv 1\), in the cases \(2n < p\), top, and \(2n >p\),  bottom}
    \label{fig:trans_mats}
\end{figure}

By \autoref{lemma:d(i)values} and \autoref{lemma:walk_on_a_graph}, a stationary distribution is
\[
    \pi_i = \frac{\min\set{i,p-i,n,p-n}}{n(p-n)}.
\]

Observe that \(\pi T = \pi\).
In particular, this stationary distribution assigns equal probability to being on an even or an odd state; that is,
\[
    \longsubsum{\sum}{i \equiv 0 \pmod*{2}}{\pi_i} = \longsubsum{\sum}{i \equiv 1 \pmod*{2}}{\pi_i} = \frac12.
\]
Thus, for \(n\) even, the chain converges to the stationary distribution, provided that the initial distribution \(\nu\) has equal weighting for even and odd states or that the chain is made lazy by taking the transition matrix to be \(\frac12 (Q+I)\).
Meanwhile, for \(n\) odd, \(\pi\) is the stationary distribution with equal weighting given to the even-state and odd-state walks.

If \(n \in \set{\frac{p-1}{2}, \frac{p+1}{2}}\), it can be shown that the eigenvalues of \(\bar{Q}\) are
\[
\set{1, -\tfrac12, \tfrac13, \ldots, (-1)^{\frac{p+1}{2}}\tfrac{2}{p-1}}.
\]
Then (by the proof of \autoref{prop:w(i)=w(p-i)_props}\ref{Q_eigenvalues}) the eigenvalues of \(Q\) are the eigenvalues in this set each with multiplicity \(2\) if \(n\) is odd, and are \(\set{\pm 1, \pm \frac12, \ldots \pm\frac{2}{p-1}}\) if \(n\) is even.
Then by \autoref{prop:w(i)=w(p-i)_props}\ref{mixing-time}, the mixing time of the walk is bounded by
\[
    t_{\textnormal{mix}}(\epsilon) \leq 
        2 \log\left(\tfrac{p^2-1}{4\epsilon}\right).
\]
\end{example}

\begin{example}\label{eg:weighted_by_dimension}
Suppose \(w(i)=i\) for each \(i\); that is, each module has a chance of being chosen proportional to its dimension.
Then for fixed \(i\) we have
\begin{align*}
\longsubsum{\sum}{ij \in E(\G)}{j}
    &= \text{(number of neighbours of \(i\))} \times \text{(average value of the neighbours of \(i\))} \\
    &= d(i) \times \mean\setbuild{j}{\simpledivide{j}{i}{n}}.
\end{align*}
If \(i+n \leq p\), all of the composition factors of \(V_i \tensor V_n\) are summands, and so their average dimension is \(\max\set{i, n}\), the midpoint of the \((i,n)\)-string or the \((n,i)\)-string (as appropriate).
If \(i+n \geq p\), the midpoint of the relevant section of the string is instead
\[
    \frac{ (\abs{i-n}+1) + (2p-(i+n-1)) }{2} = p - \min\set{i, n}.
\]
Also, by \autoref{lemma:d(i)values},
\[
d(i) = \begin{cases}
\min\set{i, n} & \text{ if \(i+n \leq p\),} \\
p - \max\set{i, n} & \text{ if \(i+n \geq p\).}
\end{cases}
\]
Thus
\begin{align*}
\longsubsum{\sum}{ij \in E(\G)}{j}
    &= \begin{cases}
    d(i) \max\set{i, n}        & \text{ if \(i+n \leq p\),}\\
    d(i) (p-\min\set{i, n})    & \text{ if \(i+n \geq p\)}
\end{cases} \\
    &= \begin{cases}
    in          & \text{ if \(i+n \leq p\),}\\
    (p-i)(p-n)  & \text{ if \(i+n \geq p\).}
\end{cases}
\end{align*}
Then
\[
    Q_{i,j} = \begin{cases}
    \frac{j}{in} & \text{if \(i+n \leq p\) and \(A_{i,j} \neq 0\),} \\
    \frac{j}{(p-i)(p-n)} & \text{if \(i+n \geq p\) and \(A_{i,j} \neq 0\),} \\
    0 & \text{otherwise.}
    \end{cases}
\]

It can be shown that \(\sum_{i \in [p-1]} \sum_{ij \in E(\G)} j = \frac{1}{6} np(p-n)(2p-n)\).
Then by \autoref{lemma:walk_on_a_graph} a stationary distribution is
\[
\pi_i = \begin{dcases}
    \frac{6i^2}{p(p-n)(2p-n)} & \text{ if \(i+n \leq p\),} \\
    \frac{6i(p-i)}{np(2p-n)} & \text{ if \(i+n \geq p\).}
\end{dcases}
\]

Now \(w(i) \neq w(p-i)\) (for all \(i\)), and so we do not have that the walk is invariant under the map \(i \mapsto p-i\). In particular, the two irreducible chains when \(n\) is odd are not isomorphic.
\end{example}

\section*{Acknowledgements}

The author would like to thank Mark Wildon (Royal Holloway, University of London) for his suggestion of this topic and his guidance whilst researching it.

This is a post-peer-review, pre-copyedit version of an article published in Algebras and Representation Theory.
The final authenticated
version is available online at \href{https://doi.org/10.1007/s10468-021-10034-0}{https://doi.org/10.1007/s10468-021-10034-0}.

\bibliographystyle{alpha}
\bibliography{references}

\end{document}